\documentclass[12pt,reqno]{amsart}
\usepackage{amsmath, amsthm, amssymb, color}
\usepackage[colorlinks=true,linkcolor=blue,urlcolor=blue,citecolor=teal]{hyperref}
\usepackage{graphicx}
\usepackage{caption}
\usepackage{bigstrut}
\usepackage{mathtools}
\usepackage{verbatim}
\usepackage{tikz}
\usepackage[all]{xy}
\usepackage[english]{babel}
\usepackage[letterpaper,top=2cm,bottom=2cm,left=2cm,right=2cm]{geometry}
\usepackage{amssymb}
\usepackage{cleveref} 
\usepackage{graphicx}
\usepackage{wasysym}
\usepackage{tikz-cd}

\usetikzlibrary{arrows}
\usepackage{algorithm}
\usepackage{shuffle}
\usepackage[normalem]{ulem}

\newtheorem{theorem}{Theorem}[section]
\newtheorem{proposition}[theorem]{Proposition}
\newtheorem{lemma}[theorem]{Lemma}
\newtheorem{corollary}[theorem]{Corollary}
\theoremstyle{definition}
\newtheorem{definition}[theorem]{Definition}
\newtheorem{remark}[theorem]{Remark}

\newtheorem{example}[theorem]{Example}

\newcommand{\NN}{\mathbb{N}}
\newcommand{\KK}{\mathbb{K}}
\newcommand{\TT}{\mathbb{T}}

\newcommand{\shu}{\shuffle}

\newcommand{\C}{\mathbb{C}}
\newcommand{\N}{\mathbb{N}}
\newcommand{\R}{\mathbb{R}}

\newcommand{\G}{\mathcal{G}}

\newcommand{\ot}{\otimes}

\newcommand{\s}{\mathop{\mathbb{S}}\nolimits}

\newcommand{\GL}{\mathop{\rm GL}\nolimits}

\newcommand{\Sym}{\mathop{\rm Sym}\nolimits}

\newcommand{\rk}{\mathop{\rm rk}\nolimits}

\newcommand{\Lie}{\mathop{\rm Lie}\nolimits}

\newcommand{\im}{\mathop{\rm im}\nolimits}

\newcommand{\compositions}{P}
\newcommand{\rd}[1]{\left\lfloor #1\right\rfloor}
\newcommand{\ru}[1]{\left\lceil #1\right\rceil}

\newcommand{\f}{\varphi}

\title{Rank and symmetries of signature tensors}
\author[F. Galuppi]{Francesco Galuppi}
\address[F. Galuppi]{Faculty of Mathematics, Informatics, and Mechanics, University of Warsaw, Banacha 2, 02-097 Warsaw, Poland}
\thanks{Galuppi acknowledges support by the National Science Center, Poland, under the projects ``Complex contact manifolds and geometry of secants'', 2017/26/E/ST1/00231, and ``Tensor rank and its applications to signature tensors of paths'', 2023/51/D/ST1/02363.}
\email{galuppi@mimuw.edu.pl (ORCID 0000-0001-5630-5389)}

\author[P. Santarsiero]{Pierpaola Santarsiero}
\address[P. Santarsiero]{Universit\`a di Bologna, Dipartimento di Matematica, Piazza di Porta S.Donato 5, Bologna, Italy}
\email{pierpaola.santarsiero@unibo.it (ORCID 0000-0003-1322-8752)}
\thanks{Santarsiero was partially supported by the Deutsche Forschungsgemeinschaft (DFG, German Research Foundation) -- Projektnummer 445466444 and by the European Union under NextGenerationEU. PRIN 2022, Prot. 2022E2Z4AK}

\begin{document}

\begin{abstract}
The signature of a path is a sequence of tensors which allows to uniquely reconstruct the path. 
In this paper we propose a systematic study of basic properties of signature tensors, starting from their rank, symmetries and conciseness. We prove a sharp upper bound on the rank of signature tensors of piecewise linear paths. We show that there are no skew-symmetric signature tensors of order three or more, and  
we also prove that specific instances of partial symmetry can only happen for tensors of order three.
Finally, we give a simple geometric characterization of paths whose signature tensors are not concise. 
\end{abstract}

\maketitle

\section{Introduction}\label{section: intro}

The signature of a path $X$ is a sequence of tensors that encodes all the essential information contained in $X$.  
Signatures have been introduced in \cite{Chenoriginal} for paths $X:[0,1]\to\R^d$ under the assumption that $X$ is sufficiently smooth, namely that each component $X_1,\dots,X_d$ of $X$ has bounded variation.\footnote{Some readers might prefer to replace the bounded variation assumption with any other hypothesis that guarantees that the integral \eqref{defin di signature con gli iterated integrals} is well defined, for instance taking $X$ to be continuous and piecewise differentiable. As pointed out in \cite[Section 2]{DiehlReizenstein}, this definition of signature tensors can be extended to many different classes of paths, including Brownian motion.}
For every positive integer $k$ it is possible to define a tensor $\sigma^{(k)}(X)\in (\R^d)^{\ot k}$ whose $(i_1,\dots,i_k)$-th entry is the iterated integral
\begin{equation}\label{defin di signature con gli iterated integrals}
\sigma^{(k)}(X)_{i_1,\dots,i_k}=\int_{0}^{1}\int_{0}^{t_k}\dots\int_{0}^{t_3}\int_{0}^{t_2}\dot{X}_{i_1}(t_1)\cdots\dot{X}_{i_k}(t_k)dt_1\cdots dt_k.
\end{equation}
By convention we set $\sigma^{(0)}(X)=1$. The \emph{signature} of $X$ is the sequence
$\sigma(X)=(\sigma^{(k)}(X)\mid k\in\N)$.
By \cite[Theorem 4.1]{chenuniqueness}, $\sigma(X)$ allows to uniquely reconstruct $X$, up to a mild equivalence relation. Signatures and their uniqueness have been extended in \cite{HL10} and \cite{uniquenessrough} to much more general classes of paths. Just like moment tensors can encode probability distributions, the ability of signature tensors to efficiently store 
all the features of the path made them prominent in stochastic analysis, and in particular in the theory of rough paths \cite{lyons2014rough}.  

Recently, signatures are starting to attract the attention of researchers in very different areas of mathematics. For instance, topological data analysts found several ways to turn persistent homology barcodes into paths, and use signatures to extract their features. These methods match or outperform state-of-the-art techniques in texture, orbit or shape recognition  \cite{signaturesetopdata}. It is interesting to remark that in this approach the path is quickly forgotten, and the main focus is on its signature tensors.
Signature techniques also proved to be efficient and reliable in areas as diverse as medical statistics \cite{medical}, quantum field theory \cite{quantumfield}, 
machine learning \cite{machine} and cybersecurity \cite{cyber}, just to mention a few. The aforementioned examples show a growing presence of signature tensors in modern applied sciences and motivated us to start a deeper investigation of these objects from an algebraic viewpoint.

The interest of signature tensors in algebraic geometry started with \cite{AFS19}. However, many of their basic algebraic properties have not been studied yet. In this paper we focus on their \emph{rank}, \emph{symmetries} and \emph{conciseness}. 
The rank of a tensor $T$ is a 
natural number $\rk(T)$ that measures the complexity of $T$. We propose a systematic study of ranks of signature tensors, and the starting point is the following version of \cite[Corollary 7.3]{AGRSS}.

\begin{theorem}\label{thm: caratterizzazione segmenti AGRSS}Let $X$ be a path of bounded variation. The following are equivalent:
\begin{enumerate}
    \item $X$ is a segment.
\item $\rk(\sigma^{(k)}(X))=1$ for every $k\ge 1$.
\item $\sigma^{(k)}(X)$ is symmetric for every $k\ge 1$.
\end{enumerate}
\end{theorem}
These statements can be considered prototypes, but they indicate a pattern: a very specific condition on the rank - all signatures of rank 1 - corresponds to a very specific family of paths, namely segments.
The purpose of this paper is to generalize \Cref{thm: caratterizzazione segmenti AGRSS} in several directions. In this process we would like to bring 
tools from multilinear algebra into signature tensor theory, and to contribute to the conversation between the signature community and the tensor decomposition community.

Since segments form a vary small class of paths, in \Cref{section: piecewise linear} we extend our study to \emph{piecewise linear paths}. This class is relatively simple, so piecewise linear paths are useful as a testing ground for the new theory. At the same time, the importance of their signatures is specifically addressed in \cite{lyonsxu}, \cite[Section 5]{AFS19} and \cite[Section 6]{PSS19}. Moreover, in applications people sometimes use \textit{time series} instead of paths \cite{DiehlReizenstein, moretimeseries}. A time series is a finite list of vectors of data, so effectively it is a piecewise linear path.  
For all these reasons, they are the natural place where to begin our investigation. In \Cref{theorem: bound rango tutte signatures} we give a bound on their rank by providing an explicit decomposition.

In \Cref{section: pure_volume} we consider a different generalization  of segments: the class of \emph{pure volume paths}. This will allow us to move a few steps in the world of rough paths, that is a central area of modern stochastic analysis \cite{frizhairer}.
While very different than segments, in \Cref{thm: pure volume} we show that their signatures enjoy some similar properties.

Looking back to \Cref{thm: caratterizzazione segmenti AGRSS}, it may be surprising to see the equivalence between the second and the third statement. This is a special feature of signatures tensors: 
according to \cite[Theorem 7.1]{AGRSS}, a signature is symmetric if and only if its rank is 1. In \Cref{section: symmetries} we relax the symmetry assumption and we look at \emph{partially symmetric} tensors, namely those which are invariant under permuting a subset of the indices, as well as \emph{skew-symmetric} tensors. In \Cref{prop: result_skew_signature} and \Cref{coroll: if k ge 4 ogni parzialmente simmetrico è simmetrico.} we show that these instances are also rare, as they may happen only at low levels of the signature.

In the tensor decomposition community, it is natural to study a tensor in its smallest possible embedding, or in other words to assume that it is \emph{concise}. However, signatures are not abstract tensors, in the sense that they come with a given embedding. 
For this reason it makes sense to ask what paths have concise signature tensors. The answer is \Cref{propos: coinciseness for paths}, that links conciseness of the signature to the fact that the image of $X$ lies on an hyperplane.

\subsection*{Significance for applications} Rank, symmetry and conciseness are most basic tensor  properties, so it is natural to study these features of signature tensors. In addition, our results also have the potential to improve and speed up computations with signature tensors, which are the objects that ultimately play the crucial role in applications. Most notably, rank is essential for data compression: storing a tensor in $(\R^d)^{\ot k}$
requires storing $d^k$
numbers, so it quickly grows infeasible when $k$ is large. On the other hand, storing a rank $r$ tensor only requires $d r k$ numbers. Thus knowing upper bounds on the rank allows us to beat the curse of dimensionality and to perform faster computations on such tensors. Moreover, in all results of \Cref{section: piecewise linear} we not only give a sharp upper bound on the rank, but we also provide a way to write down a minimal decomposition, that is tipically convenient for practical computations. We refer to \cite[Chapter 7]{Hackbusch2019} for a discussion on the importance of tensor rank from the algorithmic and numerical viewpoint. 

Likewise, exploiting symmetries is a standard way to reduce the number of parameters of a problem and make computations more efficient. Moreover, the kind of partial symmetry we study can only happen for signatures of order 3 by  \Cref{coroll: if k ge 4 ogni parzialmente simmetrico è simmetrico.}. This is especially relevant when it comes to numerically reconstruct the path from the third signature, like in \cite{PSS19}. Knowing when the third signature is partially symmetric has the potential to speed up such algorithms.

The same could be said of conciseness: in practical situations it is very convenient to work in the smallest tensor space, which is equivalent to perform the standard Tucker decomposition \cite[Chapter 8]{Hackbusch2019}. However, we would like to stress further potential consequences of our approach. \Cref{propos: coinciseness for paths} show that the signature of $X$ detects whether $X$ lies on a linear subspace. We believe that similar techniques will allow us to detects whether $X$ lies on a sphere or on some higher degree hypersurface. This would have considerable interest for the use of signatures in the data analysis community \cite{signaturesetopdata}.

\section{Shuffle identity and Thrall modules}\label{section: prelim}
Since the pioneering \cite{Chenoriginal} and \cite{chen1957integration}, signatures have been employed to study paths in $\R^d$. However, most of the arguments can be easily generalized to any field of characteristic 0. Since in this paper we use some tools from representation theory, sometimes we work over $\C$. For this reason we give our definitions over a field $\KK\in\{\R,\C\}$, and we write explicitly when we make a more precise assumption on the field. As we shall see, we prove that all our results hold over both $\mathbb{R}$ and $\mathbb{C}$.

The purpose of this section is to introduce the algebraic tools we use to study signature tensors, namely the shuffle identity and the Thrall modules. For the former, the textbook reference is \cite{Reutenauer}, while the latter have been introduced in \cite{AGRSS}. 

Let $V$ be a $\KK$-vector space of dimension $d$. If $X:[0,1]\to V$ is a 
path of bounded variation, then its signature is an element of the \emph{tensor algebra}
\[
\TT((V))=\KK\times V\times V^{\ot 2}\times V^{\ot 3}\times \cdots .
\]
This can be interpreted as the graded $\KK$-algebra of formal power series in $d$ non-commuting variables, and the multiplication is the tensor product. Elements of $\TT((V))$ are sequences of tensors. If $T\in\TT((V))$ and $i_1,\dots,i_k\in\{1,\dots,d\}$, then we denote by $T_{i_1,\dots,i_k}$ the entry $(i_1,\dots,i_k)$ of the order $k$ element of $T$. We consider $i_1,\dots,i_k$ as \emph{letters} in the alphabet $\{1,\dots,d\}$. A \emph{word} is an ordered sequence of letters. If $v$ and $w$ are words of the same length, then we formally define
\[T_{v+w}=T_v+T_w.\] By convention, the $0$-th entry of the signature of a path is 1, so it is convenient for us to define $$\TT_1((V))=\{T \in \TT((V))\mbox{ with constant entry } 1\}.$$
Analogously, we set $\TT_0((V))=\{T \in \TT((V)) \mbox{ with constant entry } 0\}.$
The tensor algebra has a remarkably rich structure. For instance, it admits a Lie bracketing
$
[T,S]=T\ot S-S\ot T$.
\begin{definition}\label{def: lie algebra}
We denote by $\Lie(V)$ the Lie subalgebra of $\TT_0((V))$ generated by $V$. In other words, it is the smallest vector subspace of $\TT_0((V))$ that contains $V$ and is closed under bracketings. 
If we set $\Lie^k V= \Lie(V) \cap V^{\ot k}$, then
$$\Lie(V)= \bigoplus_{k \in \NN} \Lie^k V$$
is a graded vector space.
Denote by $\Lie((V))= \KK\times\Lie^1(V)\times \Lie^2(V)\times\dots\subset \TT_0((V))$.
\end{definition}
By using this power series notation, we can define the exponential of elements of the tensor algebra.
\begin{definition}\label{definition: exp}
Define $\exp:\TT_0((V))\to \TT_1((V))$ by the formal power series
\[\exp(T)=\sum_{n=0}^{\infty}\frac{T^{\otimes n}}{n!}.\]
The image of $\Lie((V))$ under $\exp$ is denoted by $\G(V)$.
\end{definition}
There is a more combinatorical definition of $\G(V)$, for which we need to introduce another operation between words. This is called shuffle, and it plays an important role for signatures.
We denote by $v\cdot w$ the word obtained by writing $v$ followed by $w$.
\begin{definition}
The \textit{shuffle product} of two words $v$ and $w$, denoted by $v\shuffle w$, is the formal sum of all order-preserving interleavings of them. More precisely, the shuffle is defined recursively. If $i$ and $j$ are letters, then
\begin{align*}
(v\cdot i) \shuffle (w\cdot j) = \left(v\shuffle (w\cdot j)\right)\cdot i + \left((v\cdot i)\shuffle w\right)\cdot j.
\end{align*}
As an example, $12\shuffle 34=1234+1324+1342+3124+3142+3412$.
\end{definition}
The relation between this operation and signature tensors is the \emph{shuffle identity}. 

\begin{lemma}\label{lem:shuffle identity}
If $\sigma$ is the signature of a bounded variation path $X:[0,1]\to V$, then $\sigma
_v\cdot\sigma
_w= \sigma
_{v\shuffle w}	$ for all words $v$ and $w$.
\end{lemma}
The shuffle identity had a crucial role in the proof of \Cref{thm: caratterizzazione segmenti AGRSS} and will be one of our main tools in this paper because it gives relations among entries at different levels of a signature. It also provides a characterization of elements of $\G(V)$. Indeed,
\[
\G(V)=\{T\in \TT_1((V))\mid T_v\cdot T_w= T_{v\shuffle w}	\mbox{ for all words $v$ and $w$}\}.
\]
Since the signature of a  
path of bounded variation belongs to $\G(V)$, we can think of an element of $\G(V)$ as the generalization of the signature of a path. This is also the reason why elements of $\Lie((V))$ are called \emph{log-signatures}. Indeed, we will phrase most of our results in terms of log-signatures, rather than in terms of paths.

Besides the shuffle identity, in this paper we will systematically employ the \emph{Thrall decomposition}, introduced in \cite[Section 3]{AGRSS}, that will allow us to apply some basic techniques from representation theory. First, we need to write down the exponential map with some more details.
\begin{definition}\label{definition: phi_k}\label{definition: Thrall modules}
Let $p_k$ be the projection $\TT_1((V)) \to V^{\ot k}$ onto the $k$-th factor. Define $\f_k=p_k\circ \exp:\Lie((V)) \to V^{\ot k}$ to be 
\begin{align}
\f_k((T_{(i)}\mid i\in\N))= \sum{\frac{1}{t!} T_{(\alpha_1)} \ot \cdots \ot T_{(\alpha_{t})}},
\end{align}
where the sum is over all tuples of positive integers
$(\alpha_1,\ldots,\alpha_{t})$ such that $\alpha_1+\cdots +\alpha_{t} = k$. Now we decompose $\f_k$ as a sum of functions indexed by partitions of $k$. If $\lambda\vdash k$, define $f_{\lambda}: \Lie((V)) \to
     V^{\ot k}$ by
\begin{align}\label{eq: map f_lambda}
    f_{\lambda}((T_{(i)}\mid i\in\N))= \sum_{\alpha \in \compositions(\lambda)}\frac{1}{t!}{T_{(\alpha_1)} \ot \cdots \ot T_{(\alpha_{t})}},\nonumber
\end{align}
where $P(\lambda)$ is the set of distinct permutations of $\lambda$. Partitions of $k$ also 
allow us to define the \emph{Thrall modules}: if $a_i(\lambda)$ is the number of times the integer $i$ appears in $\lambda$, then let
\[W_{\lambda}(V) = \Sym^{a_1(\lambda)}(\Lie^1(V)) \otimes \cdots \otimes \Sym^{a_{k}(\lambda)}(\Lie^k(V)).\]
\end{definition}
By \cite[Theorem 3.2]{AGRSS} we know that the Thrall modules decompose the tensors space
\[V^{\ot k}= \bigoplus_{\lambda \vdash k}{W_{\lambda}(V)}\]
as a direct sum of $\GL(V)$-representations. The proof is given over $\C$, but it carries over to $\R$ as well, see \cite[Section 8.5]{Reutenauer}. 
There are many questions about the Thrall decomposition, for instance how it compares to the classical Schur decomposition
. For us, the Thrall decomposition is a useful instrument to deal with signatures, so we refer to \cite{AGRSS} for a more general discussion. However, just to give a taste of the interaction between Thrall modules and signatures, now we show how most of the $W_\lambda(V)$ do not contain signatures.

\begin{theorem}\label{theorem: empty_thrall_characterization}
Let $\ell=(T_{(i)}\mid i\in\N)\in\Lie((V))$ be a log-signature and let $\lambda=(\lambda_1,\dots,\lambda_s)$ be a partition of $k$ such that $\lambda$ has two distinct entries. Supose that $\f_k(\ell)\in W_\lambda(V)$. If there exists $i\in\{1,\dots,s\}$ such that $\lambda_i\mid k$, then $\f_k(\ell)=0$.
\end{theorem}
\begin{proof}
Since $\lambda_i\mid k$, we can consider the partition $\mu=(\lambda_i,\lambda_i,\dots,\lambda_i)$ of $k$. By hypothesis $\lambda\neq\mu$. Since $\f_k(\ell)\in W_\lambda(V)$ and the Thrall modules are in direct sum, we deduce that $f_\mu(\ell)=0$. But $f_\mu(\ell)$ is a power of $T_{(\lambda_i)}$, hence $T_{(\lambda_i)}=0$. In order to conclude that $\f_k(\ell)=0$, it is enough to observe that each 
summand of $\varphi_k(\ell)=f_\lambda(\ell)$ is a multiple of $T_{(\lambda_i)}$, because $\lambda_i$ is one of the entries of $\lambda$.
\end{proof}

\begin{example}\label{example: partition (2,3,6)}
The converse of \Cref{theorem: empty_thrall_characterization} does not hold. Consider for instance $k=11$ and $\lambda=(2,3,6)$. While no entry of $\lambda$ divides $k$, there is no nonzero signature in $W_{(2,3,6)}$. Indeed, assume by contradiction that 
there exists a log-signature 
$\ell=(T_{(1)},T_{(2)},\dots)\in\Lie((V))$ 
such that 
$\f_{11}(\ell)\in W_{(2,3,6)}\setminus \{0\}$. This means that $0\neq \f_{11}(\ell)= f_{(2,3,6)}(\ell)$, and so 
$T_{(2)}$, $T_{(3)}$ and $T_{(6)}$ 
are nonzero. But then $\f_{11}(\ell)$ has a nonzero component in several different Thrall modules, because for instance $f_{(2,3,3,3)}(\ell)$ is nonzero. That contradicts the fact that $\f_{11}(\ell)\in W_{(2,3,6)}$. 
\end{example}

We close this section by recalling the action of $\GL(V)$ on signature tensors.

\begin{remark}\label{remark: possiamo agire con GL}
The natural action of $\GL(V)$ on $V$ extends to an action of $\GL(V)$ on $V^{\ot k}$ by
\[
M\cdot (v_1\ot\cdots\ot v_k)= (M\cdot v_1)\ot\cdots\ot(M\cdot v_k).
\]
The group $\GL(V)$ also acts on paths by composition, namely if $X$ is a path and $M\in\GL(V)$, then we can define another path $M\cdot X:[0,1]\to V$ by
\[t\mapsto M\cdot (X_1(t),\dots,X_d(t)).\] 
By \cite[Theorem 3.1]{chen1957integration}, 
$\sigma^{(k)}(M\cdot X)$ and $\sigma^{(k)}(X)$ have the same rank. See also \cite[Lemma 3.3]{DiehlReizenstein} for a modern reference.
\end{remark}

\section{Piecewise linear paths}\label{section: piecewise linear}
In this section we study what happens when, in the statement of \Cref{thm: caratterizzazione segmenti AGRSS}, we replace the class of segments with the larger class of piecewise linear paths. These are more complicated that segments, and their signatures will have rank greater than 1. Nevertheless, piecewise linear paths can still  be considered relatively simple, so we expect their rank to be small.

\begin{definition}\label{def: decomposition and rank}
A tensor $T\in (\KK^d)^{\ot k}$ is called {\it elementary} if there exist $v_1,\ldots,v_k\in \KK^d$ such that $T=v_1\ot\cdots\ot v_k$. A \textit{decomposition} of $T$ is a way to write
 \begin{equation*}\label{eq: tensor decomposition}
T= T_1+\cdots +T_r
 \end{equation*}
as a sum of elementary tensors. The \textit{rank} of $T$, denoted by $\rk(T)$, is the minimum length of a decomposition of $T$. 
\end{definition}

As pointed out in \cite[Section 8]{nphard}, in general it is very difficult to compute the rank of a given tensor, so in this section we provide an upper bound. In some small cases we employ flattening methods to actually compute the exact value. The rank of a tensor may well depend on the field $\KK$. In this paper we provide bounds on the rank by exhibiting explicit decompositions, hence our results are independent on the ground field. 

\begin{remark}\label{rmk: many possible definitions of rank}
\Cref{def: decomposition and rank} goes back to \cite{hitchcock}, but the last century has seen the introduction of several other possible generalizations of the matrix rank. Examples include the \emph{border rank} \cite[Section 1.2.6]{Lan}, the \emph{multilinear rank} \cite[Definition 2.3.1.4]{Lan}, the \emph{geometric rank} \cite{geometricrank} and the \emph{tensor train rank} \cite{tensortrain}. It would certainly be worth it to study different ranks of signature tensors. In this paper we deal with the rank as in \Cref{def: decomposition and rank}.
\end{remark}
By definition, a piecewise linear path is a concatenation of segments. 

\begin{definition}\label{definition: segment}
If $v\in V$ is a vector, we can consider the path $X:[0,1]\to V$ given by
\begin{equation}\label{equation: segment}
    t\to tv.
\end{equation}
By \cite[Theorem 4.1]{chenuniqueness}, the signature does not change under translation nor reparametrization of $X$, so we define the \emph{segment} associated to $v$ to be a smooth path equivalent to \eqref{equation: segment}, up to translation or reparametrization. Applying \eqref{defin di signature con gli iterated integrals}, it is an easy exercise to see that the signature of this segment is 
\begin{equation}\label{eq: signature of a segment}
\left(\frac{v^{\ot k}}{k!}\mid k\in\N\right).
\end{equation}
\end{definition}

\begin{definition}\label{def: concatenation pw linear path} If $X,Y:[0,1]\to V$ are 
continuous paths, then their \emph{concatenation} is the path
\[(X\sqcup Y)(t)=\begin{cases}X(2t) & \mbox{ if } t\in \left[0,\frac12\right],\\X(1)-Y(0)+Y(2t-1) & \mbox{ if } t\in \left[\frac12,1\right].\end{cases}\]
A path $X$ is \emph{piecewise linear} if there exist segments $X_1,\dots,X_m$ such that $X=X_1\sqcup\dots\sqcup X_m$.
\end{definition}

The formula to compute the signature of the concatenation of two paths, known as Chen's identity, was proven in \cite[Theorem 2.1]{chen1957integration} for path in $\R^d$ and extended in \cite[Theorem 5.1]{chen1957integration} to paths in $\C^d$. It is a striking link between paths and the structure of the tensor algebra.

\begin{theorem}\label{them: concatenation theorem} If $X,Y:[0,1]\to V$ are paths of bounded variation, then $\sigma(X\sqcup Y)=\sigma(X)\ot\sigma(Y)$.
\end{theorem}

If $X_1,\dots,X_m$ are segments, then by \Cref{them: concatenation theorem} the $k$-th signature of a piecewise linear path is 
\begin{equation}\label{eq: weak naif bound on the rank of a pw linear path}
    \sigma^{(k)}(X_1\sqcup\dots\sqcup X_m)=\sum_{
a_1+\cdots+a_m=k}
\sigma^{(a_1)}(X_1)\ot\cdots\ot\sigma^{(a_m)}(X_m).
\end{equation}
By \Cref{thm: caratterizzazione segmenti AGRSS}, each of the $\binom{m+k-1}{k}$ summands of \eqref{eq: weak naif bound on the rank of a pw linear path} is a rank one tensor. This explicit decomposition tells us that $\rk(\sigma^{(k)}(X_1\sqcup\dots\sqcup X_m))\le \binom{m+k-1}{k}$. Hence, for a fixed value of $k$, the rank is bounded by a polynomial of degree $k$ in the variable $m$. The rest of the section is devoted to improving this bound to a polynomial of degree $k-2$. We start by dealing with low values of $m$.

\begin{proposition}\label{prop:rank_2pieces}\label{propos: m=2}Let $X_1,X_2$ be segments and call $u=\sigma^{(1)}(X_1)$ and $v=\sigma^{(1)}(X_2)$. If $k=2s+1$ is odd, then
\begin{align}\label{eq:odd_signature_2pieces}
\sigma^{(2s+1)}(X_1\sqcup X_2)= \sum_{j=0}^{s} \frac{u^{\otimes 2j}}{(2j)!} \otimes   \left(\frac{u}{2j+1}+\frac{v}{2s-2j+1} \right)
   \otimes \frac{v^{\otimes 2s-2j}}{(2s-2j)!}.
\end{align}
If $k=2s$ is even, then
\begin{align}\label{eq:even_signature_2pieces}
\sigma^{(2s)}(X_1\sqcup X_2)=\sum_{j=1}^s \frac{u^{\otimes 2s-2j+1}}{(2s-2j+1)!}\otimes \left( \frac{u}{2s-2j+2}+\frac{v}{2j-1}\right)\otimes \frac{v^{\otimes 2j-2}}{(2j-2)!}+ \frac{v^{\ot k}}{k!}.
\end{align}
In particular,
$\rk (\sigma^{(k)}(X_1\sqcup X_2))\le \lceil \frac{k+1}{2}\rceil$ and equality holds if and only if $u$ and $v$ are linearly independent.
\end{proposition}
\begin{proof}
By \eqref{eq: signature of a segment} we know that $\sigma^{(i)}(X_1)=\frac{1}{i!}u^{\otimes i}$ and $\sigma^{(i)}(X_2)=\frac{1}{i!}v^{\otimes i}$.
\Cref{them: concatenation theorem} gives
\begin{align*}
\sigma^{(k)}(X_1\sqcup X_2)= \sum_{j=0}^{k} \sigma^{(j)}(X_1)\otimes \sigma^{(k-j)}(X_2)
=\sum_{j=0}^{k} \frac{1}{j!}u^{\otimes j}\otimes \frac{1}{(k-j)!}v^{\otimes k-j}.
\end{align*}
If $k=2s+1$ is odd, then $\sigma^{(2s+1)}(X_1\sqcup X_2)$ is the sum of $2s+2$ rank one tensors that we can group together in pairs to obtain \eqref{eq:odd_signature_2pieces}.
If $k=2s$ is even, again we can group in pairs all the rank one tensors except the last one and obtain \eqref{eq:even_signature_2pieces}.
In both cases we have a sum of $s+1$ rank one tensors, that gives the required bound on the rank.  
Now, if $X_1$ and $X_2$ are not linearly independent, then $X_1\sqcup X_2$ is a segment and its signatures all have rank 1 by \Cref{thm: caratterizzazione segmenti AGRSS}. Hence we assume that they are independent and we prove that decompositions \eqref{eq:odd_signature_2pieces} and \eqref{eq:even_signature_2pieces} are minimal. 

Observe that $\sigma^{(k)}(X_1\sqcup X_2)\in \langle u,v\rangle^{\ot k}$, so, in order to study its rank, it is not restrictive to assume that $d=2$. In order to prove that $\rk (\sigma^{(k)}(X_1\sqcup X_2))\ge \ru{\frac{k+1}{2}}$, we will employ \emph{flattenings}. Flattening a tensor $T\in V_1\ot\dots\ot V_k$ means choosing two disjoint sets of indices $I_1$ and $I_2$ such that $\{1,\dots,k\}=I_1\cup I_2$ and look at $T$ as the matrix  in $$\left(\bigotimes_{i\in I_1}V_i\right)\ot \left(\bigotimes_{j\in I_2}V_j\right)\mbox{ associated to corresponding linear map } \bigotimes_{j\in I_2}V_j^*\to \bigotimes_{i\in I_1}V_i.$$
As explained in \cite[Section 3.4]{Lan}, the rank of the matrix is a lower bound on the rank of $T$. Assume that $k=2s+1$ is odd. Notice that in \eqref{eq:odd_signature_2pieces} the term $\frac{u}{2j+1}+\frac{v}{2s-2j+1}$ always appears in an odd factor, so we consider $I_1$ and $I_2$ to be the set of odd and even indices, respectively. 
The associated matrix $M$ has $2^{s}$ columns, indexed by all the basis elements of $\langle u^*,v^*\rangle^{\ot s}$. By construction, $M$ has $s+1$ nonzero columns, the one corresponding to the basis elements $\{(u^*)^{\ot j}\ot (v^*)^{\ot s-j}\mid j\in\{0,\dots,s\}\}$. These $s+1$ column vectors
are
\[
\left\{\frac{u^{\otimes j}}{(2j)!} \otimes   \left(\frac{u}{2j+1}+\frac{v}{2s-2j+1} \right)
   \otimes \frac{v^{\otimes s-j}}{(2s-2j)!}\mid j\in\{0,\dots,s\}
\right\}.
\]
These vectors are linearly independent, hence $\rk(\sigma^{(2s+1)}(X_1\sqcup X_2))\ge\rk(M)=s+1$. We treat the case $k=2s$ in a similar way. Observe that in \eqref{eq:even_signature_2pieces} the term $\frac{u}{2s-2j+2}+\frac{v}{2j-1}$ always appears in an even factor, so we switch the roles of $I_1$ and $I_2$ and we flatten $\sigma^{(2s)}(X_1\sqcup X_2)$ as a linear map
\[\bigotimes_{i\mbox{\tiny{ odd}}}\langle u,v\rangle^* \to \bigotimes_{i\mbox{\tiny{ even}}}\langle u,v\rangle.\]
As before, the matrix $M$ associated to this linear map has $s+1$ nonzero columns
\begin{align*}
\left\{
\frac{u^{\otimes s-j}}{(2s-2j+1)!}\otimes \left( \frac{u}{2s-2j+2}+\frac{v}{2j-1}\right)\otimes \frac{v^{\otimes j-1}}{(2j-2)!}\mid j\in\{1,\dots,s\}
\right\}\cup \left\{\frac{v^{\ot s}}{(2s)!} \right\},
\end{align*}
which are linearly independent. Hence this flattening gives the required bound $\rk(\sigma^{(2s)}(X_1\sqcup X_2))\ge s+1$.
\end{proof}

\begin{example}We illustrate \Cref{propos: m=2} for $k=5$. 
We decompose
\begin{align*}
\sigma^{(5)}(X_1\sqcup X_2)&=\frac{u^{\otimes5}}{5!}+\frac{u^{\otimes 4}\otimes v}{4!}+\frac{u^{\otimes 3}\otimes v^{\otimes 2}}{3!2!}+\frac{u^{\otimes 2}\otimes v^{\otimes 3}}{2!3!}+\frac{u\otimes v^{\otimes 4}}{4!}+\frac{v^{\otimes 5}}{5!}\\&
=\frac{u^{\otimes 4}}{4!}\otimes \left(\frac{u}{5}+v\right)+\frac{u^{\otimes 2}}{2!}\otimes \left(\frac{u}{3}+\frac{v}{3}\right)\otimes \frac{v^{\otimes 2}}{2!}+\left(u+\frac{v}{5}\right)\otimes \frac{v^{\otimes 4}}{4!}
\end{align*}
as a sum of $3=\ru{\frac{5+1}{2}}$ elementary tensors. In each summand of the decomposition we collect the first, third and fifth factors on one side, and the remaining two factors on the other side. In other words we flatten $\sigma^{(5)}(X_1\sqcup X_2)$ as the $8\times 4$ matrix
\begin{align*}
M&=\left(\frac{u^{\otimes 2}}{4!}\otimes \left(\frac{u}{5}+v\right)\right)\ot
u^{\otimes 2}+\left(
\frac{u}{2!}\otimes \left(\frac{u}{3}+\frac{v}{3}\right)\otimes \frac{v}{2!}\right)\ot
(u\otimes v)+\left(\left(u+\frac{v}{5}\right)\otimes \frac{v^{\otimes 2}}{4!}\right)\ot
v^{\otimes 2},
\end{align*}
associated to the linear map $\langle u^*, v^*\rangle^{\ot 2}\to \langle u, v\rangle^{\ot 3}$ sending
\begin{center}
    \begin{tabular}{lll}
$u^*\otimes u^* \mapsto \frac{u^{\otimes 2}}{4!}\otimes \left(\frac{u}{5}+v\right)$     &$\quad$ &  $u^*\otimes v^*\mapsto\frac{u}{2!}\otimes\left(\frac{u}{3}+\frac{v}{3}\right)\otimes\frac{v}{2!}$\\
$v^*\otimes u^*\mapsto 0$ & $\quad$& $v^*\otimes v^*\mapsto\left(u+\frac{v}{5}\right)\otimes\frac{v^{\otimes 2}}{4!}$.
\end{tabular}
\end{center}
Since the vectors $\frac{u^{\ot 2}}{4!}\ot\left(\frac{u}{5}+v\right)$, $\frac{u}{2!}\ot\left(\frac{u}{3}+\frac{v}{3}\right)\ot\frac{v}{2!}$ and $\left(u+\frac{v}{5}\right)\ot\frac{v^{\ot 2}}{4!}$ are linearly independent, the flattening gives $\rk(\sigma^{(5)}(X_1\sqcup X_2))\ge \rk(M)= 3$. 
\end{example}

\begin{proposition}\label{propos: piecewise linear m=3}
Let $X_1,X_2,X_3$ be segments and let us call $u=\sigma^{(1)}(X_1)$, $v=\sigma^{(1)}(X_2)$ and $w=\sigma^{(1)}(X_3)$.
If $k=2s+1$ is odd, then 
\begin{align*}
    \sigma^{(k)}(X_1\sqcup X_2\sqcup X_3)=&\sum_{i=0}^s \frac{u^{\otimes 2i}}{(2i)!}\otimes \left( \frac{u}{2i+1} + v +\frac{w}{2(s-i)+1} \right)\otimes \frac{w^{\otimes 2(s-i)}}{(2(s-i))!}\\
    &+\sum_{i=0}^{s-1}\left( \sum_{j=1}^{2(s-i)-1} \frac{u^{\otimes j}}{j!}\otimes \frac{v^{\otimes 2(s-i)-j}}{(2(s-i)-j)!}\otimes \left(\frac{v}{2(s-i)-j+1}+ \frac{w}{2i+1} \right)\otimes \frac{w^{\otimes 2i}}{(2i)!}   \right)
    \\ &+\sum_{i=1}^s \frac{v^{\otimes 2i}}{(2i)!}\otimes \left( \frac{v}{2i+1}+ \frac{w}{2(s-i)+1} \right)\otimes \frac{w^{\ot 2(s-i)}}{(2(s-i))!}.
\end{align*}
If $k=2s$ is even, then 
\begin{align*}
    \sigma^{(k)}(X_1\sqcup X_2\sqcup X_3)=&\sum_{i=0}^{s-1} \frac{u^{\otimes 2i+1}}{(2i+1)!}\otimes \left( \frac{u}{2i+2} + v +\frac{w}{2(s-i-1)+1} \right)\otimes \frac{w^{\otimes 2(s-i-1)}}{(2(s-i-1))!}\\
&+\sum_{i=0}^{s-2}\left( \sum_{j=1}^{2(s-i)-2} \frac{u^{\otimes j}}{j!}\otimes \frac{v^{\otimes 2(s-i)-j-1}}{(2(s-i)-j-1)!}\otimes \left(\frac{v}{2(s-i)-j}+ \frac{w}{2i+1} \right)\otimes \frac{w^{\otimes 2i}}{(2i)!}   \right)
    \\ &+\sum_{i=0}^{s-1} \frac{v^{\otimes 2i+1}}{(2i+1)!}\otimes \left( \frac{v}{2i+2}+ \frac{w}{2(s-i)-1} \right)\otimes \frac{w^{2(s-i-1)}}{(2(s-i-1))!}+\frac{w^{\otimes k}}{k!}.
\end{align*}
In particular, $\rk \sigma^{(k)}(X_1\sqcup X_2\sqcup X_3)\le \ru{\frac{(k+1)^2}{4}}$.
\begin{proof}
As in \eqref{eq: weak naif bound on the rank of a pw linear path} we can write
\begin{equation}\label{eq:3pieces_all_k}
   \sigma^{(k)}(X_1\sqcup X_2\sqcup X_3)=\sum_{a_1+a_2+a_3=k} \frac{u^{\otimes \alpha_1}}{a_1!}\otimes \frac{v^{\otimes \alpha_2}}{a_2!}\otimes \frac{w^{\otimes \alpha_3}}{a_3!}
\end{equation}
as the sum of $\binom{k+2}{2}$ elementary tensors, corresponding to the non-commuting monomials of degree $k$ in which the variables $u$, $v$ and $w$ appear in this order. Assume that $k=2s+1$ is odd and call $A_1,A_2,A_3$ the three summands of the formula in our statement. The number of non-commuting monomials appearing in $A_1$ is $3(s+1)$, in $A_2$ we recognize $2s^2$ of these monomials and $2s$ more appear in $A_3$
. Observe that they are all distinct, and that the coefficient of each monomial equals the one in decomposition \eqref{eq:3pieces_all_k}. This means that all the $3(s+1)+2s^2+2s=\binom{2s+3}{2}$ monomials appearing in \eqref{eq:3pieces_all_k} also appear in $A_1+A_2+A_3$. Similarly, if $k=2s$ is even and we call $B_1,B_2,B_3$ the summands in the formula of our statement, we recognize $3s+2s(s-1)+2s$ elementary tensors and one more given by $w^{\otimes k}/k!$. This means that all the $3s+2s(s-1)+2s+1=\binom{2s+2}{2}$ monomials appearing in \eqref{eq:3pieces_all_k}   also appear in $B_1+B_2+B_3+w^{\otimes k}/k!$.

To conclude, notice that each summand of $A_1,A_2,A_3$ and $B_1,B_2,B_3$ is an elementary tensor, from which we deduce the bound on the rank.
\end{proof}
\end{proposition}
The next example shows that the bound from \Cref{propos: piecewise linear m=3} is sharp for $k=3$.
\begin{example}\label{example: for k=3 or m=3 our bound is sharp}
Consider the segments $X_1$, $X_2$ and $X_3$ corresponding to the vectors $e_1=(1,0,0)$, $e_2=(0,1,0)$ and $e_3=(0,0,1)$. This is called \emph{canonical axis path} in \cite[Example 2.1]{AFS19}. If we denote $e_{ijk}=e_i\ot e_j\ot e_k$, then $\sigma^{(3)}(X_1\sqcup X_2 \sqcup X_3)$ is
\[
\frac16e_{111}+\frac12e_{112}+\frac12e_{113}+\frac16e_{222}+\frac12e_{223}+\frac12e_{122}+e_{123}+\frac12e_{133}+\frac12e_{233}+\frac16e_{333}\in (\KK^3)^{\ot 3}.\]
We know that $\rk(\sigma^{(3)}(X_1\sqcup X_2 \sqcup X_3))\le 4$ and we want to prove that equality holds.
Since the only possible flattenings of $\sigma^{(3)}(X_1\sqcup X_2 \sqcup X_3)$ are $6\times 3$ matrices, they are not sufficient. For this reason we employ 
\emph{Koszul flattenings} (called \emph{Strassen equations} in \cite[Section 2.1]{LanOtt15}). While the usual flattening of a tensor $T\in U\otimes V\otimes W$ as an element of $U\otimes (V\otimes W)$ is the matrix corresponding to a linear map $T_U:U^*\to V\ot W$, the Koszul flattening is the composition 
$$
F: U^*\otimes W \xrightarrow{T_U\otimes \mathrm{id}_W} V\otimes W\otimes W \xrightarrow{\mathrm{skew}}V\otimes \wedge^2W,
$$
where $\mathrm{skew}(v\ot w_1\ot w_2)=v\ot (w_1\ot w_2- w_2\ot w_1)$. By \cite[Theorem 2.1]{LanOtt15} we get $\rk(T)\ge \frac{\rk(F)}{\dim(W)-1}$. So in order to prove that $\rk (\sigma^{(3)}(X_1\sqcup X_2 \sqcup X_3) )=4$, we just check directly that the rank of its Koszul flattening is 7.
\end{example}

Now that we have sorted out the case $m=2$ and $m=3$, we consider small values of $k$.

\begin{proposition}\label{propos: minimal decomposition for second and third signature}\label{propos: k=2}
If $X_1,\dots,X_m$ are segments, then  
$$
\sigma^{(2)}(X_1\sqcup \cdots \sqcup X_m)=\sum_{i=1}^m \sigma^{(1)}(X_i) \otimes \left( \frac{1}{2} \sigma^{(1)}(X_i)+\sigma^{(1)}(X_{i+1})+\cdots +\sigma^{(1)}(X_m)\right).
 $$   
In particular, $\rk(\sigma^{(2)}(X_1\sqcup \cdots \sqcup X_m))\le m$ and equality holds if $X_1,\dots,X_m$ are linearly independent.
\end{proposition}
\begin{proof}
\Cref{them: concatenation theorem} gives 
\begin{align*}
\sigma^{(2)}(X_1\sqcup \cdots \sqcup X_m)=&\sum_{i=1}^m\sigma^{(2)}(X_i)+\sum_{1\le k<\ell\le m}\sigma^{(1)}(X_i)\otimes \sigma^{(1)}(X_j)\\
=&\sum_{i=1}^m \frac{1}{2}\sigma^{(1)}(X_i)^{\otimes 2}+\sum_{1\le k<\ell\le m}\sigma^{(1)}(X_i)\otimes \sigma^{(1)}(X_j)\\
 =&\sum_{i=1}^m \sigma^{(1)}(X_i) \otimes \left( \frac{1}{2} \sigma^{(1)}(X_i)+\sigma^{(1)}(X_{i+1})+\cdots +\sigma^{(1)}(X_m)\right).
\end{align*}
If $X_1,\dots,X_m$ are independent, then by \Cref{remark: possiamo agire con GL} we can assume that they correspond to the first $m$ element of the standard basis of $V$. As highlighted also in \cite[Example 2.3]{AFS19},  $\sigma^{(2)}(X_1\sqcup \cdots \sqcup X_m)$ is an upper triangular matrix of size $d\times d$ with $m$ nonzero diagonal entries.
\end{proof}

Next we move to the case $k=3$. Instead of just bounding the rank of the third signature of a piecewise linear path, we produce a decomposition for a slightly larger class of order 3 tensors, that will be useful later.
\begin{definition}\label{definition: T_kaplha}
Let $k$ and $\alpha$ be integers such that $k\geq 2$ and $\alpha\geq 0$. If $v_1,\dots,v_m\in V$, then we define
$$S_{k,\alpha}(v_1,\dots,v_m)=\sum_{a_1+\cdots+a_m=k}  \frac{v_1^{\otimes a_1}\otimes v_2^{\otimes a_2}\otimes \cdots \otimes v_m^{\otimes a_m}}{(a_1+\alpha)!a_2!\cdots a_m! }\in V^{\ot k}.$$
\end{definition}
Notice that if $X_1,\dots,X_m$ are segments, then $S_{k,0}(\sigma^{(1)}(X_1),\dots,\sigma^{(1)}(X_m))=\sigma^{(k)}(X_1\sqcup\dots\sqcup X_m)$.

\begin{proposition}\label{proposition: k=3 rank bound}
   Let $v_1,\dots,v_m\in V$ and let $s=\lceil \frac{m}{2}\rceil$. Then
\begin{align*}
S_{3,\alpha}(v_1,\dots,v_m)  =&\; v_1^{\otimes 2}\otimes \left( \frac{v_1}{(3+\alpha)!}+\frac{v_2}{(2+\alpha)!}+\cdots + \frac{v_m}{(2+\alpha)!} \right) \\
&+\sum_{i=2}^{s}v_i^{\otimes 2}\otimes \left( \frac{v_i}{3!}+\frac{v_{i+1}}{2}+ \cdots +\frac{v_m}{2}\right)\\
&+\sum_{i=2}^s \left( \frac{v_1}{(\alpha+1)!}+v_2+\cdots + v_{i-1} \right)\otimes v_i\otimes  \left( \frac{v_i}{2}+v_{i+1}+\cdots  + v_m\right)\\
&+\sum_{i=s+1}^{m}\left( \frac{v_1}{(\alpha+1)!2}+\frac{v_2}{2}+\cdots +\frac{v_{i-1}}{2})+\frac{v_i}{3!} \right)\otimes v_i^{\otimes 2} \\
&+\sum_{i=s+1}^{m-1}\left( \frac{v_1}{(\alpha+1)!}+v_2+\cdots + v_{i-1}+\frac{v_i}{2} \right) \otimes v_i \otimes \left(  \sum_{j=i+1}^{m}  v_j \right).
\end{align*}
In particular,  $\rk( S_{3,\alpha}(v_1,\dots,v_m))\leq 2m-2$.
\end{proposition}
\begin{proof}
By \Cref{definition: T_kaplha}, the tensor $S_{3,\alpha}(v_1,\dots,v_m)$ is the sum of $\binom{3+m-1}{3} $ elementary tensors that correspond to degree three non-commuting monomials in which the variables $v_1,\dots,v_m$ appear in this order. We will prove the result by comparing the formula in our statement with the one defining $S_{3,\alpha}(v_1,\dots,v_m)$. Call $A_1,\dots,A_5$ the summands appearing on the right hand side of our statement. The number of elementary tensors appearing in $A_1$ is $m$, in $A_2$ we recognize $\frac{(2m-s)(s-1)}{2}$ elementary tensors, there are $ \frac16 s(s-1)(3m-2s+1)$ of them in $A_3$, $\frac12 (m-s)(m+s+1) $ 
 in $A_4$, and
 $\frac16 (m-s)(m^2+s(m-3)-2s^2-1) $  
of them in $A_5$. The monomials appearing in $A_1,\dots,A_5$ are all distinct and the coefficient of each monomial equals the one in the decomposition of $S_{3,\alpha}(v_1,\dots,v_m)$ given via  \Cref{definition: T_kaplha}. Hence, all 
 \begin{align*}
 &m+\frac{(2m-s)(s-1)}{2}+\frac16 s(s-1)(3m-2s+1)+\frac12 (m-s)(m+s+1)\\ &+\frac16 (m-s)(m^2+s(m-3)-2s^2-1)
 =\binom{3+m-1}{3}
 \end{align*}
monomials appearing in $A_1,\dots,A_5 $ also appear in $S_{3,\alpha}(v_1,\dots,v_m)$ and this concludes the proof. 
\end{proof}

\Cref{example: for k=3 or m=3 our bound is sharp} shows that the bound given by \Cref{proposition: k=3 rank bound} is sharp for $\alpha=0$ and $m=3$. Now we clarify why we used $S_{3,\alpha}(v_1,\dots,v_m)$ instead of $\sigma^{(3)}(X_1\sqcup\dots\sqcup X_m)$.

\begin{remark}\label{remark: T kalpha si spezza per fare induzione}
\Cref{definition: T_kaplha} allows us to argue by induction. Indeed,
\begin{align*}
  S_{k,\alpha}(v_1,\dots,v_m) &=   v_1\ot \sum_{\substack{ a_1+\cdots+a_m=k\\a_1\geq 1}} \frac{v_1^{\otimes a_1-1}\otimes \cdots \otimes v_m^{\otimes a_m}}{(a_1+\alpha)!a_2!\cdots a_m! }+ \frac{1}{\alpha!} \sum_{a_2+\cdots+a_m=k}  \frac{v_2^{\otimes a_1}\otimes \cdots \otimes v_m^{\otimes a_m}}{a_2!\cdots a_m! }\\
  &= v_1\ot \sum_{b_1+\cdots+b_m=k-1}  \frac{v_1^{\otimes b_1}\otimes \cdots \otimes v_m^{\otimes b_m}}{(b_1+1+\alpha)!b_2!\cdots b_m! }+ \frac{1}{\alpha!}S_{k,0}(v_2,\dots,v_m)\\
  &=v_1\ot S_{k-1,\alpha+1}(v_1,\dots,v_m) +\frac{1}{\alpha!}S_{k,0}(v_2,\dots,v_m).
\end{align*}
In particular, $\rk (S_{k,\alpha}(v_1,\dots,v_m))\leq \rk(S_{k-1,\alpha+1}(v_1,\dots,v_m))+\rk(S_{k,0}(v_2,\dots,v_m))$.
\end{remark}

Before we give our improved bound, we need a small result that will simplify its statement.
\begin{lemma}\label{lemma: somme di somme di somme}
If $k\ge 4$ and $m\geq 4$, then
$$
\sum_{a_1=4}^m \sum_{a_2=4}^{a_1} \cdots \sum_{a_{k-3}=4}^{a_{k-4}} (a_{k-3}-1)= 
\binom{m+k-6}{k-2}+2\binom{m+k-7}{k-3}.
$$
\end{lemma}

\begin{proof}
Let $s=k-3$. We prove that  
 $$
 \sum_{a_1=4}^m \sum_{a_2=4}^{a_1} \cdots \sum_{a_{s}=4}^{a_{s-1}} (a_{s}-1)= 
 \binom{m+s-3}{s+1}+2\binom{m+s-4}{s}
$$
 by induction on $s$.
For $s=1$ we have $\sum_{a_1=4}^m(a_1-1) =\frac{(m-3)(m+2)}{2!}=\binom{m-2}{2}+2\binom{m-3}{1}$. If $s\geq 2$, by induction hypothesis we have
\begin{align*}
\sum_{a_1=4}^m \left(\sum_{a_2=4}^{a_1} \cdots \sum_{a_{s}=4}^{a_{s-1}} (a_{s}-1)\right)
&=\sum_{a_1=4}^m\binom{a_1+s-4}{s}+2\sum_{a_1=4}^m\binom{a_1+s-5}{s-1}\\
&=\sum_{a=0}^{m-4}\binom{a+s}{s}+2\sum_{a=0}^{m-4}\binom{a+s-1}{s-1}\\
&=\sum_{a=0}^{m+s-4-s}\binom{a+s}{s}+2\sum_{a=0}^{m+s-5-(s-1)}\binom{a+s-1}{s-1}\\
&=\binom{m+s-4+1}{m+s-4-s}+2\binom{m+s-5+1}{m+s-5-(s-1)},
\end{align*}
where the last equality is the so-called hockey stick identity.
\end{proof}

\begin{theorem}\label{theorem: bound rango tutte signatures}
 Let $k,m,\alpha$ be integers such that $k\geq 3$, $m\geq 4$ and $\alpha \geq 0$. If $v_1,\dots,v_m\in V$, then $\rk(  S_{k,\alpha}(v_1,\dots,v_m))$ is at most
 $$
 \sum_{j=0}^{k-4} \binom{m+j-4}{j} \ru{\frac{(k-j+1)^2}{4}}
 +
 2\binom{m+k-6}{k-2}+4\binom{m+k-7}{k-3}.
 $$
\end{theorem}
\begin{proof}
We argue by induction on $k$. Assume that $k=3$. Keeping in mind that any sum on the empty set gives 0, we have to show that $\rk(S_{3,\alpha}(v_1,\dots,v_m))\le 2m-2$, which we have done in \Cref{proposition: k=3 rank bound}. Now we assume that $k\ge 4$ and we repeatedly employ \Cref{remark: T kalpha si spezza per fare induzione} to bound
\begin{align*}
\rk&(S_{k,\alpha}(v_1,\dots,v_m))\le \rk(S_{k,0}(v_2\dots,v_m))+\rk(S_{k-1,\alpha+1}(v_1,\dots,v_m))\\
&\le \rk(S_{k,0}(v_3\dots,v_m))+\rk(S_{k-1,1}(v_2\dots,v_m))+\rk(S_{k-1,\alpha+1}(v_1,\dots,v_m))\\
& \hskip10pt\vdots \\
&\le \rk(S_{k,0}(v_{m-2},v_{m-1},v_m))+\left(\sum_{a_0=4}^{m-1}\rk(S_{k-1,1}(v_{m-a_0+1},\dots,v_m))\right)+\rk(S_{k-1,\alpha+1}(v_1,\dots,v_m)).
\end{align*}
By \Cref{propos: piecewise linear m=3}, we know that $\rk(S_{k,0}(v_{m-2},v_{m-1},v_m))\le \ru{\frac{(k+1)^2}{4}}$. Now we apply the induction hypothesis on $\rk(S_{k-1,1}(v_{m-a_0+1},\dots,v_m))$ and $\rk(S_{k-1,\alpha+1}(v_1,\dots,v_m))$ to deduce that $\rk(S_{k,\alpha}(v_1,\dots,v_m))$ is bounded from above by
\begin{align*}
\ru{\frac{(k+1)^2}{4}}
+\sum_{a_0=4}^{m}\left(
\sum_{j=0}^{k-5} \binom{a_0+j-4}{j} \ru{\frac{(k-j)^2}{4}}+2\binom{m+k-6}{k-2}+4\binom{m+k-7}{k-3}
\right).
\end{align*}
By \Cref{lemma: somme di somme di somme}, we rewrite the sum on the right as
\begin{align*}
\sum_{a_0=4}^{m}&\left(
\sum_{j=0}^{k-5} \binom{a_0+j-4}{j} \ru{\frac{(k-j)^2}{4}}+2\sum_{a_1=4}^{a_0}\dots\sum_{a_{k-4}=4}^{a_{k-5}}(a_{k-4}-1)
\right)\\
&=
\sum_{j=0}^{k-5}\sum_{a_0=4}^{m} \binom{a_0+j-4}{j} \ru{\frac{(k-j)^2}{4}}+2\sum_{a_0=4}^{m}\sum_{a_1=4}^{a_0}\dots\sum_{a_{k-4}=4}^{a_{k-5}}(a_{k-4}-1)\\
&=
\sum_{j=0}^{k-5}\binom{m+j-3}{j+1} \ru{\frac{(k-j)^2}{4}}+2\sum_{a_0=4}^{m}\sum_{a_1=4}^{a_0}\dots\sum_{a_{k-4}=4}^{a_{k-5}}(a_{k-4}-1)\\
&=
\sum_{l=1}^{k-4}\binom{m+l-4}{l} \ru{\frac{(k-l+1)^2}{4}}+2\sum_{a_0=4}^{m}\sum_{a_1=4}^{a_0}\dots\sum_{a_{k-4}=4}^{a_{k-5}}(a_{k-4}-1).
\end{align*}
By applying \Cref{lemma: somme di somme di somme} again, we obtain that $\rk(S_{k,\alpha}(v_1,\dots,v_m))$ is at most
\begin{align*}
\ru{\frac{(k+1)^2}{4}}+
\sum_{l=1}^{k-4}\binom{m+l-4}{l} \ru{\frac{(k-l+1)^2}{4}}
+2\binom{m+k-6}{k-2}+4\binom{m+k-7}{k-3}.
\end{align*}
We conclude by noticing that $\ru{\frac{(k+1)^2}{4}}=\binom{m+0-4}{0}\ru{\frac{(k+1)^2}{4}}$.
\end{proof}

\begin{remark} For $\alpha=0$, \Cref{theorem: bound rango tutte signatures} gives an upper bound on the rank of the $k$-th signature of a piecewise linear path with $m$ steps. More than that, 
if we combine it with \Cref{remark: T kalpha si spezza per fare induzione} we also get an explicit decomposition with that length. Finally we observe that, for a fixed $k$, the bound from \Cref{theorem: bound rango tutte signatures} is polynomial in $m$ and its leading term is $\frac{2m^{k-2}}{(k-2)!}$.
\end{remark}

\section{Pure volume paths}
\label{section: pure_volume}
In the last decades, the theory of signature advanced way beyond the study of  
bounded variation paths to include the far more general class of \emph{rough paths}. A precise definition of rough paths exceeds the scope of this paper. In this section, we deal with a simple prototype of rough path, the so-called \emph{pure volume path}. A first example can be found in \cite[Exercise 2.17]{frizhairer}: it is constructed as the limit of smaller and smaller loops. The authors call it a \emph{pure area path} because it has properties that a piecewise smooth path cannot have: as described in \cite[Remark 27]{roughveronese}, such a signature suggests a constant path that nevertheless encompass a nonzero area.
\begin{definition}\label{definition: pure volume path}
Let $n\in\N$. A log-signature $\ell=(T_{(i)}\mid i\in \N)\in\Lie((V))$ is a \emph{pure $n$-volume path} if $T_{(i)}=0$ for every $i\neq n$.
\end{definition}

For $n=1$, \Cref{thm: caratterizzazione segmenti AGRSS} implies that pure length paths correspond to segments. 
Moreover, one could argue that just as segments are the most basic example of a regular path, pure volume paths are a basic example of rough paths. For these reasons we find it interesting to extend part of the statement of \Cref{thm: caratterizzazione segmenti AGRSS} to $n\ge 2$. Namely we prove that pure volume paths are characterized by the asymptotic behaviour of their signatures, and we give an upper bound on their rank.

\begin{theorem}\label{theorem: criterion}\label{thm: pure volume}
Let $\ell=(T_{(i)}\mid i\in \N)\in\Lie((V))$ be a log-signature
. The following are equivalent:
    \begin{enumerate}
\item\label{item: 1_criterion} $\ell$ is a pure $n$-volume path
.
\item\label{item: 2_criterion}  for every $k\ge 1$ we have $\varphi_k(\ell)=\begin{cases}
            \frac{1}{h!}T_{(n)}^{\ot h} & \hbox{if there exists $h\in\N$ such that } k=hn\\
            0 & \hbox{otherwise}.
\end{cases}$
\item\label{item: 3_criterion}  there exists $k_0>n$ such that $\varphi_k(\ell)=\begin{cases}
            \frac{1}{h!}T_{(n)}^{\ot h} & \hbox{if there exists $h\in\N$ such that } k=hn\\
            0 & \hbox{otherwise}
        \end{cases}$ \\ for every $k\ge k_0$.
    \end{enumerate}
    If any of these equivalent statements holds, then $\rk(\varphi_k(\ell)) \leq     \rk (T_{(n)})^{\frac{k}{n}}$ whenever $k$ is a multiple of $n$.
\end{theorem}
\begin{proof}
It is immediate that \cref{item: 1_criterion} implies \cref{item: 2_criterion}, which in turn implies \cref{item: 3_criterion}. Now we assume that \cref{item: 3_criterion} holds and we prove that $T_{(i)}=0$ for every $i\neq n$. Let $i\neq n$ and take $k=ai\geq k_0$ a multiple of $i$. By hypothesis all summands of $\f_k(\ell)$ are zero, except possibly the one associated to the partition $(n,n,\dots,n)$ of $k$. Since $i\neq n$ and the Thrall modules are in direct sum, the summand of $\f_k(\ell)$ associated to the partition $(i,i,\dots,i)$ vanishes. This means that $T_{(i)}^{\ot a}=0$ and so $T_{(i)}=0$. 

If any of these statements holds, then \cref{item: 2_criterion} implies the required bound on the rank.
\end{proof}

\begin{remark}\label{rmk: the ranks of the signatures of a pure volume path can jump a lot up and down}
Pure volume paths may be useful to produce examples of signatures with specific ranks. For instance, if $T_{(2)}\in\Lie^2( V)$ is a skew-symmetric matrix of rank $r$ and $\ell=(0,T_{(2)},0,0,\dots)\in\Lie(( V))$ is a pure area path, then the sequence of ranks of signatures of $\ell$ is $(0,r,0,r^2,0,...)$. This shows that the ranks can jump up and down, and the jumps can be arbitrarily big.
\end{remark}

\section{Symmetries of signature tensors}\label{section: symmetries}

Looking at \Cref{thm: caratterizzazione segmenti AGRSS}, it is natural to ask what happens if we replace the hypothesis that the signatures are symmetric with a different constraint, such as skew-symmetry or partial symmetry. \Cref{prop: result_skew_signature} and \Cref{coroll: if k ge 4 ogni parzialmente simmetrico è simmetrico.} illustrate how most signature tensors do not admit such symmetries.

We start by considering skew-symmetric signature tensors. Our strategy boils down to 
determining which Thrall module contains the irreducible Schur module $\wedge^kV$.

\begin{lemma}\label{prop: exterior_in_trall}
If $k$ is even, then $ \wedge^k (V)\subseteq W_{(2,\dots,2)}(V)$. If $k$ is odd, then $ \wedge^k (V)\subseteq W_{(2,\dots,2,1)}(V)$.
\end{lemma}

\begin{proof}
Let $k=2s$ for some positive integer $s$. The exterior algebra $\wedge^{2s}(V)$ is 
$$
\wedge^{2s}(V)=\{ T \in V^{\otimes 2s} \mid T_{i_1,\dots,i_{2s}}=\mathrm{sgn}(\rho)T_{i_{\rho(1)},\dots,i_{\rho(2s)}} \mbox{ for every } \rho \in \mathfrak{S}_{2s} \},
$$
while \Cref{definition: Thrall modules} tells us that $W_{(2,\dots,2)}(V)=\Sym^s(\Lie^2(V))=\Sym^s(\wedge^2V)$. Let $T\in\wedge^{2s}(V)$. In order to prove that $T\in\Sym^s(\wedge^2(V))$ we have to prove that 
\begin{itemize}
\item $T_{i_1,\dots,i_{2s}}=-T_{i_{\rho(1)},\dots,i_{\rho(2s)}}$ whenever $\rho$ is the transposition $(j,j+1)$ for some odd index $j$, and
\item 
$T_{i_1,\dots,i_{2s}}=T_{i_{\rho(1)},\dots,i_{\rho(2s)}}$ whenever $\rho$ is the permutation sending $(j,j+1,l,l+1)$ to $(l,l+1,j,j+1)$ for some odd indices $j$ and $l$.
\end{itemize}
The first property comes directly from the skew-symmetry of $T$, and the second one as well, because 
that permutation has a positive sign
. 
For $k=2s+1$, 
we apply what we have just proven in the even case to see that
\[
\wedge^{2s+1}(V)\subset \wedge^{2s}(V)\ot V\subset \Sym^{s}(\wedge^2(V))\otimes V=W_{(2,\dots,2,1)}(V). \qedhere
\]
\end{proof}
As simple as it may look, \Cref{prop: exterior_in_trall} is enough to prove that signature tensors are almost never skew-symmetric. Before stating the theorem, we need a little technical lemma.

\begin{lemma}\label{lemma: power_skew_matrix_not_skew}
If $A $ is a nonzero skew-symmetric matrix, then $A^{\otimes s}$ is not skew symmetric.
\end{lemma}
\begin{proof}
Let $T=A^{\otimes s}$, hence $T_{i_1,j_1,\dots,i_s,j_s}=a_{i_1,j_1}\cdots a_{i_s,j_s}$, and assume by contradiction that $T$ is skew-symmetric. Since $A\neq 0$, at least one of its entries is nonzero, say $a_{i,j}\neq 0$. Hence
$$
0\neq a_{i,j}\cdot a_{i,j}\cdots a_{i,j}=T_{i,j,i,j,i,j\dots,i,j}=T_{i,i,j,j,i,j,\dots,i,j}=a_{i,i}\cdot a_{j,j}\cdots a_{i,j}=0,
$$
which is impossible.
\end{proof}

\begin{theorem}\label{prop: result_skew_signature}
If $k\ge 3$, then there are no nonzero skew-symmetric signature tensors of order $k$.
\end{theorem}
\begin{proof}
Let $\ell=(T_{(i)}\mid i\in\N)$ be a log signature. Recall from \Cref{definition: phi_k} that 
\begin{equation}\label{eq: phi_k}
\f_k(\ell)=\sum_{\lambda\vdash k}f_\lambda(\ell),
\end{equation}
where $f_\lambda(\ell)\in W_\lambda(V)$. By \cite[Theorem 3.2]{AGRSS}, the right hand side of \eqref{eq: phi_k}  is a direct sum. 
Assume that $\varphi_k(\ell)\in \wedge^kV$. By \Cref{prop: exterior_in_trall}, only one summand of \eqref{eq: phi_k} is nonzero. If $k$ is odd, we have $\f_k(\ell)=f_{(2,\dots,2,1)}(\ell)$. In particular $f_{(1,\dots,1)}=0$, hence $T_{(1)}^{\otimes k}$ and so $T_{(1)}=0$.  As a consequence, $\f_k(\ell)=f_{(2,\dots,2,1)}(\ell)=0$. Now assume that $k=2s$ is even. Then $\f_k(\ell)=f_{(2,\dots,2)}(\ell)=\frac{1}{s!}T_{(2)}^{\ot s}$. By assumption $\f_k(\ell)$ is skew-symmetric, so \Cref{lemma: power_skew_matrix_not_skew} implies that $\f_k(\ell)$ is actually zero. 
\end{proof}

Comparing \Cref{prop: result_skew_signature} and \Cref{thm: caratterizzazione segmenti AGRSS}, we can say that for a non-zero log-signature $\ell\in\Lie((V))$ it is impossible to have $\f_k(\ell)\in\wedge^k(V)$ for every $k$, or even for every $k$ sufficiently large.

Now we move to a different kind of symmetry. A tensor is \emph{partially symmetric} if it is invariant under permuting a certain subset of its $k$ indices. For the sake of simplicity, in this paper we only deal with tensors that are invariant either under permutation the first $k-1$ indices, or under permutation the last $k-1$ indices. In other words, we consider signature tensors that belong to
\[
\Sym^{k-1}(V)\ot V \mbox{ or } V \ot\Sym^{k-1}(V).
\]
While the arguments we present to deal with these two cases are the same, it is important that the commuting indices are all but the first one or all but the last one. Indeed, if we consider a different subset of $k-1$ indices, our results are not true anymore, as we will see in \Cref{example: k=3 d=2 partially}. The reason is that permuting the factors of $V^{\ot k}$ does not preserve the property of being a signature. For instance, 
the tensor
\[\frac16 e_1\ot e_1\ot e_1+\frac12 e_1\ot e_1\ot e_2+\frac12 e_2\ot e_1\ot e_2+\frac16 e_2\ot e_2\ot e_2
\]
satisfies the equations of \cite[Example 4.13]{AFS19}, but after swapping the first two factors it no longer does.

Just as in the case of skew-symmetric tensors, our first step is to check what are the Thrall modules containing the space of partially symmetric tensors. The difference is that, unlike $\wedge^k(V)$, the $\GL(V)$-representations 
$\Sym^{k-1}(V)\ot V \mbox{ and } V \ot\Sym^{k-1}(V)$ are not irreducible. We start by recalling some facts from representation theory, that are true over the complex field $\C$. Nevertheless we will prove our result, \Cref{coroll: if k ge 4 ogni parzialmente simmetrico è simmetrico.}, for both real and complex tensors.

\begin{remark}\label{rmk: schur decomposition of partially symmetric tensors} The space of partially symmetric tensors is $\Sym^{k-1}(\C^d)\ot \C^d \cong \Sym^{k}(\C^d)\oplus\s_{(k-1,1)}(\C^d)$ as a direct sum of irreducible $\GL(\C^d)$-representations. This is a special case of the Little-Richardson rule, illustrated for instance in \cite[Exercise 6.10]{FHrep}. 
\end{remark}

Before we state the next result, we introduce the notation $
(\lambda_1^{a_1},\dots,\lambda_t^{a_t})$ to denote the partition of $a_1\lambda_1+\dots+a_t\lambda_t$ where each number  $\lambda_i$ appears $a_i$ times.

\begin{proposition}\label{proposition: schur(k-1 2) va solo in W(t 1^k-t)}
Let $k\geq 3$. Then $W_\lambda(\C^d)$ contains a copy of the Schur module $\s_{(k-1,1)}(\C^d)$ as an irreducible subrepresentation if and only if there exists $r\in\{2,\dots,k\}$ such that $\lambda=(r,1^{k-r})$.
\end{proposition}
\begin{proof}
By \cite[Formula 4.12 and Theorem 6.3(2)]{FHrep}, in the Schur-Weil decomposition of $(\C^d)^{\ot k}$ there are $k-1$ copies of $\s_{(k-1,1)}(\C^d)$, hence it is enough to prove that each Thrall module $W_{(r,1^{k-r})}(\C^d)$ contains at least one copy of $ \s_{(k-1,1)}(\C^d)$.
First consider $r=k$. If $k\notin\{4,6\}$, then 
by \cite[page 917]{klyachko1974lie},   $W_{(k)} (\C^d)$ contains a copy of every Schur module except $\Sym^k(\C^d)$ and $\wedge^k(\C^d)$. 
The cases $k\in\{4,6\}$ can be explicitely checked on \texttt{SageMath}. 

Now consider $2\leq r \leq k-1$ and let $a_{\lambda,\mu}$ be the multiplicity  of the Schur module $\s_\mu (\C^d)$ in the decomposition of $W_{\lambda}(\C^d)$ into irreducible $\GL(\C^d)$-representation.
By \cite[Lemma 2.1]{schocker}, the coefficient $a_{(r,1^{k-r}),(k-1,1)}$ factors as
$$
a_{(r,1^{k-r}),(k-1,1)}=\sum_{s\vdash r}\sum_{t\vdash k-r}C^{(k-1,1)}_{s,t}\cdot a_{(r),s} \cdot a_{(1^{k-r}),t},
$$
where $ C^{(k-1,1)}_{s,t} $ denotes the Littlewood-Richardson coefficient (see \cite[Exercise 4.43]{FHrep}). 
Note that every term 
of the sum is a non-negative integer, because both the $a_{\lambda,\mu}$ and the Littlewood-Richardson coefficients are numbers counting some algebraic objects. Hence, in order to conclude that $a_{(r,1^{k-r}),(k-1,1)}\ge 1$ it is sufficient to find a nonzero term of the above sum. Let us have a closer look at the term
$$
C^{(k-1,1)}_{(r-1,1),(k-r)}\cdot a_{(r),(r-1,1)}\cdot a_{(1^{k-r}),(k-r)},
$$
 which appears in the summation by taking $ s=(r-1,1)$ and $t=(k-r)$. Since $W_{(1^\ell)}(\C^d)=\s_{(\ell)}(\C^d)$, we have $  a_{(1^{k-r}),(k-r)}=1$. Moreover, also $  a_{(r),(r-1,1)}\neq 0$ by the above argument on $r=k$. Finally the Littlewood-Richardson coefficient $C^{(k-1,1)}_{(r-1,1),(k-r)} $ is also nonzero by the Pieri's rule (see \cite[Exercise 4.44 and formula A.7]{FHrep}) and this concludes the proof.   
\end{proof}

Now we go back to work on a finite-dimensional vector space $V$ over a field $\KK\in\{\R,\C\}$.

\begin{lemma}\label{lemma: an element of Lie_k is not partially symmetric}
Let $k\ge 3$ and let $T_{(k)}\in \Lie^k(V)$. If $T_{(k)}\in \Sym^{k-1}(V)\ot V$    or $T_{(k)}\in V\ot \Sym^{k-1}(V)$, then $T_{(k)}=0$.
\begin{proof} Assume that $T_{(k)}\in \Sym^{k-1}(V)\ot V$. The other case is analogous. Observe that $T_{(k)}$ is the $k$-th signature of the pure $k$-volume path $\ell=(0,\dots,0,T_{(k)},0,\dots)$, hence it satisfies the shuffle relations. On one hand, this implies that $\f_i(\ell)=0$ for every $i\in\{1,\dots,k-1\}$. On the other hand
\begin{align*}
0=T_{i_1}T_{i_2,\dots,i_k}=T_{i_1\shu i_2,\dots,i_k}.
\end{align*}
By hypothesis $T_{(k)}$ is partially symmetric, thus
\begin{align*}
0&=T_{i_1}T_{i_2,i_3,\dots,i_k}=T_{i_1\shu i_2,i_3,\dots,i_k}=(k-1)T_{i_1,\dots,i_k}+T_{i_2,i_3,\dots,i_k,i_1} \mbox{ and in the same way}\\
0&=T_{i_2}T_{i_1,i_3,\dots,i_k}=T_{i_2\shu i_1,i_3,\dots,i_k}=(k-1)T_{i_1,\dots,i_k}+T_{i_1,i_3,\dots,i_k,i_2} 
\end{align*}
for every $i_1,\dots,i_k\in\{1,\dots,d\}$. This means that the last index commutes with all the other indices, hence $T_{(k)}\in\Sym^k(V)=W_{(1,\dots,1)}(V)$. Since $W_{(1,\dots,1)}(V)$ and $\Lie^k(V)=W_{(k)}(V)$ are in direct sum, we conclude that $T_{(k)}=0$.
\end{proof}
\end{lemma}

Despite \Cref{lemma: an element of Lie_k is not partially symmetric}, it is possible for an element of $\Lie^k(V)$ to have different sets of symmetries. For instance, in \cite[Example 2.2]{AGRSS} we see that elements of $\Lie^3(V)$ are invariant under flipping the first and the third index.
Our next step is to show how partial symmetry propagates to the lower levels of a signature.

\begin{proposition}\label{lemma: sigmak ps implies sigmak-1 ps}
Let $\ell=(T_{(i)}\mid i\in\N)\in\Lie((V))$ be a log signature. Let $k\ge 3$ and suppose that $\f_k(\ell)\neq 0$.
\begin{enumerate}
\item If $\f_k(\ell)\in \Sym^{k-1}(V)\ot V$, then $T_{(1)}\neq 0$ and $\f_{k-1}(\ell)\in \Sym^{k-2}(V)\ot V$.
\item If $\f_k(\ell)\in V\ot\Sym^{k-1}(V)$, then $T_{(1)}\neq 0$ and $\f_{k-1}(\ell)\in V\ot\Sym^{k-2}(V)$.
\end{enumerate}
\begin{proof}
First we prove our result under the hypothesis that $\KK=\C$. The proofs of the two parts are the same, so we present the first one. We start by showing that $T_{(1)}\neq 0$. By \Cref{rmk: schur decomposition of partially symmetric tensors}, the space of partially symmetric tensors decomposes as $\Sym^{k-1}(V)\ot V \cong \Sym^{k}(V)\oplus\s_{(k-1,1)}(V)$ as a sum of irreducible $\GL(V)$-representations. By \Cref{proposition: schur(k-1 2) va solo in W(t 1^k-t)} we write
\[\f_k(\ell)=\sum_{\lambda\vdash k} f_{\lambda}(\ell)
=\sum_{t=1}^k f_{(t,1^{k-t})}(\ell)\in\bigoplus_{t=1}^k
W_{(t,1^{k-t})}.
\]
If we had $T_{(1)}=0$, then $f_{(t,1^{k-t})}(\ell)=0$ for every $t\in\{1,\dots,k-1\}$ and therefore $\f_k(\ell)=f_{(k)}(\ell)\in W_{(k)}(V)$. Since $\f_k(\ell)$ is partially symmetric by hypothesis, it would vanish by \Cref{lemma: an element of Lie_k is not partially symmetric}, against our hypothesis. Now call $\sigma=\exp(\ell)$ the signature of $\ell$. Let $I$ be a word of length $k-2$ and $j\in\{1,\dots,d\}$ be a letter. If $\tilde{I}$ is a permutation of $I$, we have to show that $\sigma_{Ij}=\sigma_{\tilde{I}j}$. Since $T_{(1)}\neq 0$, there exists $i\in\{1,\dots,d\}$ such that $\sigma_i\neq 0$. 
The shuffle relations give
\begin{align*}
\sigma_i\sigma_{Ij}=\sigma_{i\shuffle (Ij)}=\sigma_{(i\shuffle I)j}+\sigma_{I(i\shuffle j)}.
\end{align*}
Since every summand of $(i\shuffle I)j$ is a word of length $k$ and $\f_{k}(\ell)\in \Sym^{k-1}(V)\ot V$, we get $\sigma_{(i\shuffle I)j}=\sigma_{(i\shuffle \tilde{I})j}$. In the same way $\sigma_{I(i\shuffle j)}=\sigma_{\tilde{I}(i\shuffle j)}$. Hence
\begin{align*}
\sigma_i\sigma_{Ij}
=\sigma_{(i\shuffle \tilde{I})j}+\sigma_{\tilde{I}(i\shuffle j)}
=\sigma_{i\shuffle (\tilde{I}j)}
=\sigma_i\sigma_{\tilde{I}j}\Rightarrow \sigma_i(\sigma_{Ij}-\sigma_{\tilde{I}j})=0.
\end{align*}
Since $\sigma_i\neq 0$, we conclude that $\sigma_{Ij}-\sigma_{\tilde{I}j}=0$.

It is not difficult to extend the result to the case $\KK=\R$. Assume that $\ell\in\Lie((\R^d))$ and $\f_k(\ell)\in \Sym^{k-1}(\R^d)\ot \R^d$. Then $\f_k(\ell)\in \Sym^{k-1}(\C^d)\ot \C^d$, so we can apply the argument above to deduce that $T_{(1)}\neq 0$ and $\f_{k-1}(\ell)\in \Sym^{k-2}(\C^d)\ot \C^d$. Hence
\[
\f_{k-1}(\ell)\in (\R^d)^{\ot k-1}\cap (\Sym^{k-2}(\C^d)\ot \C^d)=\Sym^{k-2}(\R^d)\ot \R^d.\qedhere
\]
\end{proof}
\end{proposition}

Now we are ready to show that many of the terms in the log-signature vanish if we assume partial symmetry.
\begin{proposition}\label{propos: se parzialmente simmetrico allora la metà della log si annulla} Let $k\ge 4$ and let $\ell=(T_{(i)}\mid i\in\N)\in\Lie((V))$ be a log-signature such that $\f_k(\ell)\neq 0$.
If either $\f_k(\ell)\in \Sym^{k-1}(V)\ot V$ or $\f_k(\ell)\in V\ot \Sym^{k-1}(V)$, then $T_{(i)}=0$ for every $2\le i\le k/2$.
\begin{proof}
Notice that it is enough to prove the result under the assumption that $\KK=\C$. The case $\KK=\R$ is a direct consequence. By \Cref{lemma: sigmak ps implies sigmak-1 ps} we know that $\f_i(\ell)$ is partially symmetric for every $i\in\{3,\dots,k\}$. Hence it is enough to show that if $\f_k(\ell)$ is partially symmetric then $T_{(\rd{\frac{k}{2}})}$ is partially symmetric.

Consider the case in which $k=2h$ is even. By \Cref{proposition: schur(k-1 2) va solo in W(t 1^k-t)} we write
\[\f_k(\ell)=\sum_{\lambda\vdash k} f_{\lambda}(\ell)
=\sum_{t=1}^k f_{(t,1^{k-t})}(\ell)\in\bigoplus_{t=1}^k
W_{(t,1^{k-t})}.
\]
This means that $0=f_{(h,h)}(\ell)=T_{(h)}^{\ot 2}$ and we conclude that $T_{(h)}=0$. On the other hand, if $k=2h+1$ is odd, then $\f_{2h}(\ell)$ is partially symmetric too by \Cref{lemma: sigmak ps implies sigmak-1 ps}. Since $2h$ is even, we argue as before to conclude that $T_{(h)}=0$.
\end{proof}
\end{proposition}

\begin{corollary}\label{coroll: if k ge 4 ogni parzialmente simmetrico è simmetrico.}
Let $k\ge 4$ and let $\ell=(T_{(i)}\mid i\in\N)\in\Lie((V))$ be a log-signature such that $\f_k(\ell)\neq 0$. If either $\f_k(\ell)\in \Sym^{k-1}(V)\ot V$ or $\f_k(\ell)\in V\ot \Sym^{k-1}(V)$, then $\f_i(\ell)\in \Sym^{i}(V)$ for every $i\in\{2,\dots, k\}$.
\begin{proof}
Since $\f_k(\ell)$ is partially symmetric, \Cref{propos: se parzialmente simmetrico allora la metà della log si annulla} implies that $T_{(2)},\dots,T_{(\rd{\frac{k}{2}})}$ vanish, so $\f_i(\ell)=\frac{1}{i!}T_{(1)}^{\ot i}$ is symmetric for every $i\in\{2,\dots,\rd{\frac{k}{2}}\}$. Now set $q=\rd{\frac{k}{2}}+1$ and consider $\f_q(\ell)$. Since $T_{(2)},\dots,T_{(\rd{\frac{k}{2}})}$ are zero, we have $$\f_q(\ell)=\frac{1}{q!}T_{(1)}^{\ot q}+T_{(q)}.$$
Since $\f_k(\ell)$ is partially symmetric, $\f_q(\ell)$ is partially symmetric too by \Cref{lemma: sigmak ps implies sigmak-1 ps}. Then 
$$T_{(q)}=\f_q(\ell)-\frac{1}{q!}T_{(1)}^{\ot q}$$
is partially symmetric too, so $T_{(q)}=0$ by \Cref{lemma: an element of Lie_k is not partially symmetric}. We deduce that $T_{(i)}=0$ and $\f_i(\ell)$ is symmetric for every $i\in\{2,\dots,q\}$. Now we repeat the argument. By considering
$$\f_{q+1}(\ell)=\frac{1}{(q+1)!}T_{(1)}^{\ot q+1}+T_{(q+1)},$$
the same reasoning implies that $T_{(q+1)}=0$ and $\f_{q+1}(\ell)$ is symmetric. By iterating the argument we obtain that $T_{(i)}=0$ and $\f_i(\ell)$ is symmetric for every $i\in\{2,\dots,k\}$.
\end{proof}
\end{corollary}

Notice that the statements of \Cref{propos: se parzialmente simmetrico allora la metà della log si annulla} and \Cref{coroll: if k ge 4 ogni parzialmente simmetrico è simmetrico.} are empty for $k=3$. Indeed, it is possible for a $d\times d\times d$ signature tensor to be partially symmetric but not symmetric. We conclude the section by taking a closer look at the case $k=3$.

\begin{example}\label{example: k=3 d=2 partially}
Let $d=2$, $k=3$, and call $\sigma=\f_3(\ell)$. We want to see what happens when $\sigma\in \Sym^2\KK^2\ot \KK^2$. First of all we observe that if $\sigma$ is partially symmetric but not symmetric, then $T_{(1)}$, $T_{(2)}$ and $T_{(3)}$ are all nonzero. Indeed, if either $T_{(1)}=0$ or $T_{(2)}=0$, then the partial symmetry of $\sigma=\frac16 T_{(1)}^{\ot 3}+T_{(3)}$ implies that $T_{(3)}\in\Sym^2\KK^2\ot \KK^2$, against \Cref{lemma: an element of Lie_k is not partially symmetric}. In addition, if $T_{(3)}=0$ then one can explicitly compute that $\sigma$ must be symmetric.  Moreover, by \Cref{def: lie algebra}, we can write $T_{(2)}\in\Lie^2(\KK^2)$ and $T_{(3)}\in\Lie^3(\KK^2)$ in coordinates, see \cite[Example 2.2]{AGRSS}. In this way we can parametrize 
any $2\times 2 \times 2$ signature tensor $\sigma$ as
\begin{center}
\begin{tabular}{llll}
$\sigma_{111}=\frac16 x^3$,    &$\sigma_{112}=\frac16 x^2y+\frac12 ax+b$, &$\sigma_{121}=\frac16 x^2y-2b$,  & $\sigma_{122}=\frac16 xy^2+\frac12 ay+c$, \bigstrut\\
$\sigma_{211}=\frac16 x^2y-\frac12 ax+b$, & $\sigma_{212}=\frac16 xy^2-2c$, & $\sigma_{221}=\frac16 xy^2-\frac12 ay+c$, & $\sigma_{222}=\frac16 y^3$.
\end{tabular}
\end{center}
Now
\begin{align*}
\sigma\in\Sym^2\KK^2\ot \KK^2 \Leftrightarrow\begin{cases}
\sigma_{121}=\sigma_{211}\\
\sigma_{122}=\sigma_{212}
\end{cases}
\Leftrightarrow\begin{cases}
-2b=-\frac12 ax+b\\
\frac12 ay+c=-2c
\end{cases}\Leftrightarrow\begin{cases}
ax=6b\\
ay=-6c.
\end{cases}
\end{align*}
One possibility is $a=b=c=0$. This forces $T_{(2)}=0$ and $T_{(3)}=0$, so in this case
$\sigma$ is symmetric of rank at most 1. 
Clearly there are other solutions, like $a=b=c=1$ and $(x,y)=(6,-6)$, which is an example of a signature that is partially symetric but not symmetric. 
Surprisingly enough, this partial symmetry implies that the rank of $\sigma$ over $\C$ is 3.  This is unexpected, because the general element of $(\C^2)^{\ot 3}$ has rank 2. The rank of a $2\times 2\times 2$ tensor is well understood (see \cite[Table 10.3.1]{Lan}). To prove that $\sigma$ has rank 3 we need to show that the three flattenings of $\sigma$ have rank 2 and that $\sigma $ satisfies the equations of the so-called hyperdeterminant \cite[Chapter 14, Proposition 1.7]{GKZ}. 
As in \cite[Example 7.7]{AGRSS}, we  write the hyperdeterminant of any $2\times 2\times 2$ signature tensor as
$-\frac13\left(yb+xc\right)^{2}\left(4yb+4xc-3a^{2}\right)$.
A direct computation shows that if $\sigma\in\Sym^2\KK^2\ot \KK^2$ is partially symmetric but not symmetric, then its hyperdeterminant vanishes and all three flattenings of $\sigma$ have maximal rank 2. We stress that the converse does not hold: the signature tensor
\[
e_1\otimes e_1\otimes e_2-2e_1\otimes e_2\otimes e_1+e_2\otimes e_1\otimes e_1+e_1\otimes e_2\otimes e_2-2e_2\otimes e_1\otimes e_2+e_2\otimes e_2\otimes e_1
\]
has complex rank 3 and it is not partially symmetric with respect to the first and second index. 
We also remark that imposing $\sigma\in \KK^2\ot\Sym^2\KK^2$ gives completely analogous results.
\end{example}

\section{Conciseness}\label{section: coincise}
This section is devoted to study the smallest tensor space containing a signature tensor. A tensor $T\in V_1\otimes \cdots \otimes V_k$ is \emph{not concise} if there exist vector subspaces $W_1\subseteq V_1, \dots, W_k\subseteq V_k$, with at least one inclusion proper, such that $T\in W_1\otimes \cdots \otimes W_k$. 
By \Cref{remark: possiamo agire con GL}, for signature tensors we are interested in a more refined version of conciseness,  called \emph{symmetric conciseness} in \cite[Section 3]{PSS19}. A tensor $T\in V^{\otimes k}$ is \emph{symmetrically concise} if there is no proper vector subspace $W\subsetneq V$ such that $T\in W^{\otimes k}$. Conciseness implies symmetric conciseness, 
but the converse does not hold. We shall see that symmetric conciseness of the signature tensors is equivalent to the image of the path not being contained in an hyperplane. First we prove two intermediate results about how symmetric conciseness behaves with respect to different levels of a signature.

\begin{lemma}\label{lemma: simmetrically non-coinciseness si diffonde attraverso le signatures}
Let $\ell=(T_{(i)}\mid i\in \N)\in\Lie((V))$. If there exists a positive integer $k$ and a vector subspace $W\subsetneq V$ such that $\f_k(\ell)\in W^{\ot k}$, then $\f_t(\ell)\in W^{\ot t}$ whenever $t\mid k$.
\begin{proof}
Without loss of generality, we assume that $W$ is an hyperplane. Thanks to \Cref{remark: possiamo agire con GL}, it is not restrictive to take $W=\{x_1=0\}$.  Assume that $k=st$ and let $I$ be a word of length $t$ containing the letter 1. We want to show that the coordinate $\sigma_I$ of $\sigma^{(t)}\in V^{\otimes k}$ is zero. Let $J=I^{\shuffle s}$. The word $J$ is a sum of words of length $k$, and the letter 1 appears in every summand of $J$. By hypothesis, $\sigma_J=0$. The shuffle relations give 
\[0=\sigma_J=\sigma_{I^{\shuffle s}}=\sigma_I^s \quad \Rightarrow \quad \sigma_I=0.\qedhere
\]
\end{proof}
\end{lemma}

\begin{lemma}\label{lemma: se sono tutti non-symm-coincisi allora il W è lo stesso per tutti}
Let $\ell=(T_{(i)}\mid i\in \N)\in\Lie((V))$. Suppose that for every positive integer $k$, the tensor $\f_k(\ell)$ is not symmetrically concise. Then there exists a vector subspace $W\subsetneq V$ such that $\f_k(\ell)\in W^{\ot k}$ for every positive integer $k$.
\begin{proof} 
For each $k$, let $W_k\subsetneq V$ be a vector subspace such that $\f_k(\ell)$ is symmetrically concise in $(W_k)^{\ot k}$. In other words, we pick $W_k$ to have minimal dimension among the vector subspaces $U\subsetneq V$ such that $\f_k(\ell)\in U^{\ot k}$. 
Now we choose an index $h\ge 1$ such that $$\dim(W_h)=\max\{\dim(W_k)\mid k\ge 1\}.$$
In other words, among the 
entries of the signature of $\ell$, we choose one that is as concise as possible.
Now, our construction guarantees that $\f_k(\ell)\in (W_{k})^{\ot k}$ for all $k$. In particular, for $k=nh$, we have $\f_{nh}(\ell)\in (W_{nh})^{\ot nh}$ for every $n\ge 1$. Since $h$ divides $nh$, we apply \Cref{lemma: simmetrically non-coinciseness si diffonde attraverso le signatures} 
to obtain $\f_{h}(\ell)\in (W_{nh})^{\ot h}$. If we call
\[
W=\bigcap_{n=1}^\infty W_{nh},
\]
then $\f_h(\ell)\in W^{\ot h}$. By minimality of $W_h$, this implies that $\dim(W)\ge \dim(W_h)$. On the other hand $W\subseteq W_h$ by definition, hence
\begin{equation}\label{eq: Uh contenuto in ogni Unh}
W_h=W=\bigcap_{n=1}^\infty W_{nh}\subseteq W_{nh} \mbox{ for every } n\ge 1.
\end{equation}
Now, let $n\ge 1$. Our choice of $h$ yields $\dim(W_h)\ge \dim(W_{nh})$, hence \eqref{eq: Uh contenuto in ogni Unh} implies that $W=W_{h}=W_{nh}$.  By definition of $W_{nh}$,
\[
\f_{nh}(\ell)\in (W_{nh})^{\ot nh}= W^{\ot nh} \quad \Rightarrow \quad \f_{n}(\ell)\in W^{\ot n}\mbox{ by \Cref{lemma: simmetrically non-coinciseness si diffonde attraverso le signatures}.}  \qedhere
\]
\end{proof}
\end{lemma}

Before we give the main result of this section, recall that a \emph{tree-like excursion} is a path $X$ concatenated with the path $t\mapsto X(1-t)$. In other words, it is a path that travels along $X$ and then comes back along the same $X$. By their nature, tree-like excursions do not change the signature \cite{chenuniqueness}.

\begin{proposition}\label{propos: coinciseness for paths} Let $\ell=(T_{(i)}\mid i\in \N)\in\Lie((V))$ be a log-signature. The following are equivalent.
\begin{enumerate}
\item \label{bullet: tutti non symm coinc}for every $k\ge 2$, the tensor $\f_k(\ell)$ is not symmetrically coincise.
\item \label{bullet: tutti non symm coinc da un certo punto in poi}there exists $k_0\in \N$ such that, for every $k>k_0$, the tensor $\f_k(\ell)$ is not symmetrically coincise.
\item\label{bullet: lo stesso W per tutti} there exists an hyperplane $W\subsetneq V$ such that $\f_k(\ell)\in W^{\ot k}$ for every $k\in\N$.
\end{enumerate}
If $\exp(\ell)$ is the signature of a 
bounded variation path $X:[0,1]\to V$, then each of the previous three statement is also equivalent to
\begin{enumerate}
\item[(4)] \label{bullet: contenuto in uno spazio affine}there exists an affine hyperplane $A\subsetneq V$, parallel to $W$, such that $\im(X)\subseteq A$ (up to tree-like excursion). 
\end{enumerate}
\begin{proof} The equivalence between \eqref{bullet: tutti non symm coinc} and \eqref{bullet: tutti non symm coinc da un certo punto in poi} follows from \Cref{lemma: simmetrically non-coinciseness si diffonde attraverso le signatures}, while the equivalence between \eqref{bullet: tutti non symm coinc} and \eqref{bullet: lo stesso W per tutti} follows from \Cref{lemma: se sono tutti non-symm-coincisi allora il W è lo stesso per tutti}. Assume now that $\exp(\ell)=\sigma(X)$ is the signature of a  
path of bounded variation. Suppose that 
\eqref{bullet: lo stesso W per tutti} holds and there exists a vector subspace $W\subsetneq V$ such that $\sigma^{(k)}(X)\in W^{\ot k}$. Again by \Cref{remark: possiamo agire con GL} we assume that $W$ has equation $x_1=0$. Consider the 
bounded variation path $Y:[0,1]\to V$ defined by\[Y(t)=(0,X_2(t),\dots,X_d(t)).\]By construction $\sigma^{(k)}(Y)=\sigma^{(k)}(X)$ for every $k\in\N$. Thanks to \cite[Theorem 4.1]{chenuniqueness} or \cite[Theorem 4]{HL10}, we deduce that $Y$ is equal to $X$, up to reparametrization,  translation and tree-like excursion. In particular, the coordinate function $X_1$ of $X$ is constant, so $\im(X)$ is contained in a proper affine subspace of $ V$ parallel to $W$.

Now assume that (4) holds and let $W\subsetneq V$ be the support of $A$. Thanks to \Cref{remark: possiamo agire con GL}, we assume that $W$ has equation $x_1=0$. If we write $X(t)=(X_1(t),\dots,X_d(t))$, then by hypothesis $X_1$ is constant. Given $k\in \N$, definition \eqref{defin di signature con gli iterated integrals} tells us that 
\[\sigma^{(k)}(X)_{i_1,\dots,i_k}=0\]
whenever $1\in\{i_1,\dots,i_k\}$. This means that $\sigma^{(k)}(X)\in W^{\ot k}$, so $\sigma^{(k)}(X)$ is not symmetrically concise.
\end{proof}
\end{proposition}

Observe that if one of the first three equivalent statements of \Cref{propos: coinciseness for paths} holds, then in particular $\f_k(\ell)$ is not concise for any $k\ge 2$. It is natural to wonder if the converse holds. The next example shows that this is not the case.

\begin{example}\label{example: a pure volume path withevery signature non-coincise but at least one signature symmetrically coincise} Let $ T_{(3)}=e_1\otimes (e_2\otimes e_3-e_3\otimes e_2)-(e_2\otimes e_3-e_3\otimes e_2)\otimes e_1\in \Lie^3(\KK^3)$ 
and let $\ell=(0,0,T_{(3)},0,\dots,)\in\Lie(\KK^3)$ be the corresponding pure volume path. We want to show that no signature of $\ell$ is concise, but there exists one at least one signature of $\ell$ which is symmetrically concise. Explicit computation shows that $\f_3(\ell)=T_{(3)}$ is symmetrically concise but 
it is not concise. By \Cref{theorem: criterion}, higher signatures of $\ell$ are either 0 or powers of $T_{(3)}$. Since the power of a non-concise tensor is also non-concise, all the signatures of $\ell$ are non-concise. 
\end{example}

Explicit computations show that for every $i\in\{1,2,3\}$ there exist a symmetrically concise signature tensor of order 3 that is concise with respect to each factor except the $i$-th one.

\section*{Acknowledgments}
We would like to thank the University of Bologna and the University of Warsaw for hosting us during our research visits.

\bibliographystyle{alpha}
\bibliography{References.bib}

\newcommand{\etalchar}[1]{$^{#1}$}
\begin{thebibliography}{MKNH{\etalchar{+}}19}

\bibitem[AFS19]{AFS19}
C.~Am{\'e}ndola, P.~Friz, and B.~Sturmfels.
\newblock Varieties of signature tensors.
\newblock In {\em Forum of Mathematics, Sigma}, volume~7. Cambridge University
  Press, 2019.

\bibitem[AGR{\etalchar{+}}23]{AGRSS}
C.~Am{\'e}ndola, F.~Galuppi, {\'A}.~Ríos, P.~Santarsiero, and T.~Seynnaeve.
\newblock Decomposing tensor spaces via path signatures.
\newblock {\em preprint arXiv:2308.11571}, 2023.

\bibitem[BGLY16]{uniquenessrough}
H.~Boedihardjo, X.~Geng, T.~Lyons, and D.~Yang.
\newblock The signature of a rough path: uniqueness.
\newblock {\em Adv. Math.}, 293:720--737, 2016.

\bibitem[Bro13]{quantumfield}
F.~Brown.
\newblock Iterated integrals in quantum field theory.
\newblock In {\em Geometric and topological methods for quantum field theory},
  pages 188--240. Cambridge Univ. Press, Cambridge, 2013.

\bibitem[CFC{\etalchar{+}}21]{cyber}
T.~Cochrane, P.~Foster, V.~Chhabra, M.~Lemercier, T.~Lyons, and C.~Salvi.
\newblock Sk-tree: a systematic malware detection algorithm on streaming trees
  via the signature kernel.
\newblock In {\em 2021 IEEE International Conference on Cyber Security and
  Resilience (CSR)}, pages 35--40, 2021.

\bibitem[Che54]{Chenoriginal}
K.-T. Chen.
\newblock Iterated integrals and exponential homomorphisms.
\newblock {\em Proc. London Math. Soc. (3)}, 4:502--512, 1954.

\bibitem[Che57]{chen1957integration}
K.-T. Chen.
\newblock Integration of paths, geometric invariants and a generalized
  {B}aker-{H}ausdorff formula.
\newblock {\em Annals of Mathematics}, 65(1):163--178, 1957.

\bibitem[Che58]{chenuniqueness}
K.-T. Chen.
\newblock Integration of paths---a faithful representation of paths by
  non-commutative formal power series.
\newblock {\em Trans. Amer. Math. Soc.}, 89:395--407, 1958.

\bibitem[CNO20]{signaturesetopdata}
I.~Chevyrev, V.~Nanda, and H.~Oberhauser.
\newblock Persistence paths and signature features in topological data
  analysis.
\newblock {\em IEEE Transactions on Pattern Analysis and Machine Intelligence},
  42(1):192--202, 2020.

\bibitem[DR19]{DiehlReizenstein}
J.~Diehl and J.~Reizenstein.
\newblock Invariants of multidimensional time series based on their
  iterated-integral signature.
\newblock {\em Acta Appl. Math.}, 164:83--122, 2019.

\bibitem[FH13]{FHrep}
W.~Fulton and J.~Harris.
\newblock {\em Representation theory: a first course}, volume 129.
\newblock Springer Science \& Business Media, 2013.

\bibitem[FH20]{frizhairer}
P.~Friz and M.~Hairer.
\newblock {\em A course on rough paths}.
\newblock Universitext. Springer, Cham, 2020.
\newblock Second edition.

\bibitem[Gal19]{roughveronese}
F.~Galuppi.
\newblock The rough {V}eronese variety.
\newblock {\em Linear Algebra and its Applications}, 583:282--299, 2019.

\bibitem[GKZ94]{GKZ}
I.~M. Ge{l'}fand, M.~M. Kapranov, and A.~V. Zelevinsky.
\newblock {\em Discriminants, resultants, and multidimensional determinants}.
\newblock Mathematics: Theory \& Applications. Birkh\"{a}user Boston, Inc.,
  Boston, MA, 1994.

\bibitem[Hac19]{Hackbusch2019}
W.~Hackbusch.
\newblock {\em Tensor Spaces and Numerical Tensor Calculus}, volume~56 of {\em
  Springer Series in Computational Mathematics}.
\newblock Springer International Publishing, Cham, 2019.
\newblock Second edition.

\bibitem[Hit27]{hitchcock}
F.~L. Hitchcock.
\newblock The expression of a tensor or a polyadic as a sum of products.
\newblock {\em Journal of Mathematics and Physics}, 6(1-4):164--189, 1927.

\bibitem[HL10]{HL10}
B.~Hambly and T.~Lyons.
\newblock Uniqueness for the signature of a path of bounded variation and the
  reduced path group.
\newblock {\em Ann. of Math. (2)}, 171(1):109--167, 2010.

\bibitem[HL13]{nphard}
C.~Hillar and L.-H. Lim.
\newblock Most tensor problems are {NP}-hard.
\newblock {\em J. ACM}, 60(6), nov 2013.

\bibitem[Kly74]{klyachko1974lie}
A.~A. Klyachko.
\newblock Lie elements in the tensor algebra.
\newblock {\em Siberian Mathematical Journal}, 15(6):914--920, 1974.

\bibitem[KMZ23]{geometricrank}
S.~Kopparty, G.~Moshkovitz, and J.~Zuiddam.
\newblock Geometric rank of tensors and subrank of matrix multiplication.
\newblock {\em Discrete Anal.}, Paper No. 1, 25, 2023.

\bibitem[Lan12]{Lan}
J.~M. Landsberg.
\newblock {\em Tensors: geometry and applications}, volume 128 of {\em Graduate
  Studies in Mathematics}.
\newblock American Mathematical Society, Providence, RI, 2012.

\bibitem[LM24]{machine}
T.~Lyons and A.~McLeod.
\newblock Signature methods in machine learning.
\newblock {\em preprint arXiv:2206.14674}, 2024.

\bibitem[LO15]{LanOtt15}
J.~M. Landsberg and G.~Ottaviani.
\newblock New lower bounds for the border rank of matrix multiplication.
\newblock {\em Theory of Computing}, 11(11):285--298, 2015.

\bibitem[LX17]{lyonsxu}
T.~Lyons and W.~Xu.
\newblock Hyperbolic development and inversion of signature.
\newblock {\em Journal of Functional Analysis}, 272(7):2933--2955, 2017.

\bibitem[Lyo14]{lyons2014rough}
T.~Lyons.
\newblock Rough paths, signatures and the modelling of functions on streams.
\newblock In {\em Proceedings of the International Congress of Mathematicians},
  volume~4 of {\em Example Conference Proceedings}, pages 163--184, Seoul, June
  2014. Kyung Moon Publishers.

\bibitem[MKNH{\etalchar{+}}19]{medical}
J.~Morrill, A.~Kormilitzin, A.~Nevado-Holgado, S.~Swaminathan, S.~Howison, and
  T.~Lyons.
\newblock The signature-based model for early detection of sepsis from
  electronic health records in the intensive care unit.
\newblock In {\em 2019 Computing in Cardiology (CinC)}, pages Page 1--Page 4,
  2019.

\bibitem[Ose11]{tensortrain}
I.~V. Oseledets.
\newblock Tensor-train decomposition.
\newblock {\em SIAM Journal on Scientific Computing}, 33(5):2295--2317, 2011.

\bibitem[PSS19]{PSS19}
M.~Pfeffer, A.~Seigal, and B.~Sturmfels.
\newblock Learning paths from signature tensors.
\newblock {\em SIAM J. Matrix Anal. Appl.}, 40(2):394--416, 2019.

\bibitem[Reu93]{Reutenauer}
C.~Reutenauer.
\newblock {\em Free {L}ie algebras}, volume~7 of {\em London Mathematical
  Society Monographs. New Series}.
\newblock The Clarendon Press, Oxford University Press, New York, 1993.
\newblock Oxford Science Publications.

\bibitem[Sch03]{schocker}
M.~Schocker.
\newblock Multiplicities of higher {L}ie characters.
\newblock {\em Journal of the Australian Mathematical Society}, 75(1):9--21,
  2003.

\bibitem[SDL{\etalchar{+}}23]{moretimeseries}
C.~Salvi, J.~Diehl, T.~Lyons, R.~Preiß, and J.~Reizenstein.
\newblock A structure theorem for streamed information.
\newblock {\em Journal of Algebra}, 634:911--938, 2023.

\end{thebibliography}

\end{document}